\documentclass[reqno]{amsart}

\usepackage{latexsym}
\usepackage{mathrsfs}
\usepackage[mathscr]{eucal}
\usepackage{amscd}
\usepackage{amsxtra}
\usepackage{amssymb}
\usepackage{amsbsy}
\usepackage{graphicx}
\usepackage{amsopn}
\usepackage{verbatim}
\usepackage[T1]{fontenc}
\usepackage{mathtools}
\usepackage[dvipsnames]{xcolor}
\definecolor{BritishGreen}{rgb}{0.01, 0.26, 0.15}
\usepackage{appendix}
\usepackage{hyperref}
\usepackage{microtype}
\usepackage[english]{babel}
\hypersetup{pdfpagemode={UseOutlines},
bookmarksopen=true,bookmarksopenlevel=0,hypertexnames=false,
colorlinks=true,citecolor=black,linkcolor=black,urlcolor=black,
pdfstartview={FitV},unicode,breaklinks=true}
\usepackage{enumitem}
\usepackage[normalem]{ulem}
\usepackage{stackengine}
\stackMath

\usepackage[textwidth=15cm, textheight=23cm, centering]{geometry}

\makeatletter
\@namedef{subjclassname@2020}{%
  \textup{2020} Mathematics Subject Classification}
\makeatother

\newcommand{\defset}[2]{\left\{#1\left|~#2\right.\right\}}
\newcommand{\set}[1]{\left\{#1\right\}}

\newtheorem{theorem}{Theorem}[section]
\newtheorem{proposition}[theorem]{Proposition}
\newtheorem{corollary}[theorem]{Corollary}
\newtheorem{lemma}[theorem]{Lemma}
\newtheorem{question}[theorem]{Question}

\theoremstyle{definition}
\newtheorem{mydef}[theorem]{Definition}

\theoremstyle{remark}
\newtheorem{remark}[theorem]{Remark}

\newtheorem{claim}[theorem]{Claim}

  \def\cF{{\mathcal{F}}}
  
  \def\cL{{\mathcal{L}}}

 \def\bQ{{\mathbb{Q}}}

 \def\bZ{{\mathbb{Z}}}

\newcommand{\w}{\omega}

\newcommand{\rest}{\upharpoonright}
\newcommand{\finishclaim}{\hfill\ensuremath{_{Claim}\square}}

\newcommand{\age}{\mathsf{age}}
\newcommand{\rk}{\mathsf{rk}}
\newcommand{\embed}{\hookrightarrow}
\newcommand{\ran}{\mathrm{ran}}

\newcommand{\simup}[1]{\ensuremath{%
  \mathrel{\stackon[1pt]{\sim}{\scriptstyle #1}}%
}}

\begin{document}

\title[A rank function for Fra\"{\i}ss\'{e} classes and the rank property]{A rank function for Fra\"{\i}ss\'{e} classes and the rank property}

\author[C. L\'opez-Callejas]{Carlos L\'opez-Callejas}
\address{Department of Mathematics, Bar-Ilan University, Ramat-Gan 5290002, Israel.}
\email{lopezcc@biu.ac.il}

\author[J. Navarro-Castillo]{Jareb Navarro-Castillo}
\address{Centro de Ciencias Matem\'aticas, Universidad Nacional Aut\'onoma de M\'exico, Campus Morelia, 58089, Morelia, Michoac\'an, M\'exico.}
\address{Institut für diskrete Mathematik
und Geometrie, Technische Universität Wien, 1040, Vienna, Austria}
\email{jareb@matmor.unam.mx}


\date{}

\begin{abstract}
Given a hereditary class $\cF$ of finite relational structures, the rank function $\rk:\sigma\cF\to\w_1\cup\{\infty\}$, introduced by Kubi\'{s} and Shelah, measures how far a countable structure is from being universal within its class: $\rk(X)=\infty$ if and only if the Fra\"{\i}ss\'{e} limit embeds into $X$. We say that $\cF$ has the \emph{Rank Property} (RP) if every countable ordinal is realized as the rank of some $X\in\sigma\cF$.
 
We develop the basic theory of the rank function and establish RP for three families of classes: those satisfying the free amalgamation property and the full extension property (covering graphs, hypergraphs, and many others); finite tournaments; and finite linear orders. For the latter, we compute the rank of every countable ordinal: if $\w^{\beta_1}\cdot c_1$ is the leading Cantor normal form term of $\alpha\geq\w$, then $\rk(\alpha)=\w\cdot\beta_1+\lfloor\log_2 c_1\rfloor$.
\end{abstract}

\maketitle

\section{Introduction}

A classical problem in Fra\"{\i}ss\'{e} theory is to determine which countable structures are universal for a given hereditary class --- that is, which structures embed every finite member of the class (see, e.g., \cite{fraisse, macpherson-survey, kubis-fraisse-sequences} for general background). A natural refinement of this problem is to measure, in a graded way, \emph{how far} a countable structure is from being universal.

The classical framework for this question is Fra\"{\i}ss\'{e} theory~\cite{fraisse}: given a hereditary class $\cF$ of finite relational structures satisfying the joint embedding property and the amalgamation property, there exists a unique countable \emph{Fra\"{\i}ss\'{e} limit} that is universal and ultrahomogeneous for $\cF$. Familiar examples include the Rado graph, the rational order $(\bQ,<)$, the generic tournament, and the generic partial order~\cite{macpherson-survey}.

The Fra\"{\i}ss\'{e} limit is the most universal object in $\sigma\cF$,\footnote{Here $\sigma\cF$ denotes the class of all countable structures isomorphic to unions of chains in $\cF$.} but what about the structures that fall short of full universality? How can one measure, in a precise and graded way, the \emph{degree of universality} of a countable structure?

The \emph{rank function} $\rk:\sigma\cF\to\w_1\cup\{\infty\}$ (Definition \ref{2:definition_rankFunction}) provides a natural answer to this question: ordinals serve as a natural scale of partial universality: $\rk(X)=0$ means there are single-element structures in $\mathcal{F}$ that are not substructures of $X$,  $\rk(X)=\infty$ means that the Fra\"{\i}ss\'{e} limit embeds into $X$ (see Proposition~\ref{rk infty iff universal}), and the countable ordinal values in between form a fine hierarchy. 

The rank admits a game-theoretic interpretation (Proposition~\ref{rank via game}): two players alternately choose one-point extensions and realizations in $X$, and the rank is infinite if and only if Player~\textsf{II} can keep the game going forever.

The central question driving this paper is:
\begin{quote}
\emph{For which hereditary classes $\cF$ does the rank function achieve every countable ordinal value?}
\end{quote}
We say that $\cF$ has the \emph{Rank Property} (RP) if for every $\alpha<\w_1$ there exists $X\in\sigma\cF$ with $\rk(X)=\alpha$. The Rank Property can be understood as a fine-structure analysis of the Fra\"{\i}ss\'{e} limit and its class: it asserts that the hierarchy of partial universality is as rich as possible, with no ordinal levels missing.

Since $\rk(X)\geq\w$ forces $\cF$ to be a Fra\"{\i}ss\'{e} class (Theorem~\ref{rk geq omega implies AP}) and is equivalent to $\age(X)=\cF$ (Proposition~\ref{rk geq omega iff age}), the Rank Property forces $\cF$ to be Fra\"{\i}ss\'{e} (Corollary~\ref{RP implies Fraisse}). The converse does not hold: the class of finite sets is Fra\"{\i}ss\'{e} but fails RP. Understanding which Fra\"{\i}ss\'{e} classes have RP is a natural structural problem at the intersection of combinatorics and model theory.

We establish the Rank Property in three settings, each requiring different techniques.

\textbf{Free amalgamation classes.} If $\cF$ satisfies the \emph{free amalgamation property} (FAP) and the \emph{full extension property} (FEP), we prove RP (Theorem~\ref{every rank with extra caso general}). The proof constructs finite ``universal'' structures $H_n$ with $\rk(H_n)=n$, then builds countable structures of any prescribed rank $\gamma+n$ by amalgamating copies of previously constructed structures over $H_n$ as a \emph{kernel}. The free amalgamation property ensures that the leaves of the kernel construction are pairwise disconnected, which is essential for the upper bound on the rank. This covers the classes of finite graphs, finite hypergraphs, finite directed graphs, edge-colored structures, and, more generally, any class of finite structures in a finite relational language with no axioms beyond irreflexivity of non-unary relations\footnote{By irreflexivity we mean that $R(x_0,\dots,x_{k-1})$ implies $x_i\neq x_j$ for $i\neq j$. This generalizes the usual notion for binary relations.} (Corollary~\ref{examples FAP FEP}). In the language of Conant~\cite{conant-free-amalgamation}, such classes give rise to \emph{free amalgamation theories}, whose model-theoretic properties (NSOP$_4$, rosiness, 1-basedness when simple) have been extensively studied.

\textbf{Tournaments.} The class of finite tournaments has strong amalgamation but not FAP: every pair of vertices must receive a directed edge. Its Fra\"{\i}ss\'{e} limit is the generic countable tournament, classified by Lachlan~\cite{lachlan-tournaments}. We prove RP for tournaments (Corollary~\ref{tournament RP}) by combining a \emph{Cross-Piece Bound} on kernel amalgamations (Lemma~\ref{tournament cross piece bound}) with a \emph{rigidity property} of the structures $H_n$ (Lemma~\ref{tournament rigidity corollary}), which forces realizations of certain types to lie inside $H_n$. These ingredients replace the disconnection between components in the FAP setting; the rest of the construction is formally parallel to the kernel amalgamation for graphs.

\textbf{Linear orders.} The class of finite linear orders is a Fra\"{\i}ss\'{e} class (with $(\bQ,<)$ as its Fra\"{\i}ss\'{e} limit) but does not satisfy FAP: every amalgam must totally order the elements, then there is no ``free'' choice. We introduce the \emph{Interval Characterization} (Proposition~\ref{interval characterization}): the rank of a finite substructure $F$ in a linear order $Y$ equals the minimum rank of the intervals determined by $F$ in $Y$. This reduces rank computations to a one-dimensional problem and replaces the kernel machinery entirely. The key upper-bound tool is a \emph{Pigeonhole Lemma for intervals} (Lemma~\ref{pigeonhole for intervals}), which bounds the rank in terms of the number of convex pieces with the ``correct'' tail behavior.

For finite linear orders, the rank is completely determined by cardinality: $\rk(Y)=\lfloor\log_2(|Y|+1)\rfloor$ (Theorem~\ref{rank of finite linear orders}). For infinite ordinals, we compute the rank of every countable ordinal (Theorem~\ref{rank CNF}): if $\alpha\geq\w$ has Cantor normal form with leading term $\w^{\beta_1}\cdot c_1$, then $\rk(\alpha)=\w\cdot\beta_1+\lfloor\log_2 c_1\rfloor$. In particular, the rank depends only on the leading term and refines the classical Hausdorff rank~\cite{hausdorff-grundzuge}: $\rk(\alpha)=\w\cdot\rk^{H}(\alpha)+\lfloor\log_2 c_1\rfloor$ (Remark~\ref{rank vs Hausdorff}). We extend this to linear orders of the form $\bZ\cdot\alpha$ (Corollary~\ref{rank Z CNF}). The Rank Property for finite linear orders (Corollary~\ref{RP for linear orders}) then follows directly from these computations.

The rank function was introduced by Kubi\'{s} and Shelah in~\cite{kubis-shelah}, within the framework of \emph{evolution systems}~\cite{kubis-radecka} --- a category-theoretic generalization of Fra\"{\i}ss\'{e} theory where ``transitions'' (one-point extensions) replace the classical age/amalgamation setup. In that setting, the rank measures the universality of objects in a locally countable evolution system. The present paper originated when the authors visited Prague and learned of the rank function from Kubi\'{s}. We develop the theory from scratch in a purely combinatorial setting, without requiring the categorical machinery. The forthcoming paper by Kubi\'{s} and Shelah~\cite{kubis-shelah} develops the rank in the full generality of evolution systems.

The Rank Property is related to, but distinct from, several other rank notions. The \emph{K-rank} of Guingona and Parnes~\cite{guingona-parnes} is a rank on partial types parametrized by a strong amalgamation Fra\"{\i}ss\'{e} class; for linear orders, it coincides with dp-rank. The Hausdorff rank~\cite{hausdorff-grundzuge} on scattered linear orders measures depth in the condensation hierarchy, and our Theorem~\ref{rank CNF} shows precisely how the rank function refines it.

The article is organized as follows. Section~\ref{sec rank} introduces the rank function and establishes its basic properties: monotonicity under substructures, intermediate values, and the connection to automorphisms. Section~\ref{sec rank and AP} develops the general theory connecting the rank to the amalgamation property: rank controls embeddability, $\rk\geq\w$ characterizes Fra\"{\i}ss\'{e} classes, and $\rk=\infty$ characterizes universality. Section~\ref{FAP and FEP section} proves RP for classes with FAP and FEP via the kernel construction. Section~\ref{tournaments section} proves RP for tournaments via the Cross-Piece Lemma. Section~\ref{sec linear orders} develops the theory for linear orders: the Interval Characterization, the Rank Property, the rank of countable ordinals, and the rank of $\bZ\cdot\alpha$. Section~\ref{sec questions} collects open questions for further investigation.
\section{The rank function and its basic properties}\label{sec rank}

Fix $\cL$, a countable relational language without constants, and $\cF$ a countable up to isomorphisms class of $\cL$-structures, which is also hereditary and closed under isomorphisms.

In this section every structure will be an $\mathcal{L}$-structure. By $\sigma\cF$, we denote the class of all countable structures isomorphic to unions of $\subseteq$-chains in $\cF$. Usually, $A,B,C$, etc., denote structures in $\cF$, while $X,Y,Z$, etc., denote structures in $\sigma\cF$.\footnote{As usual, when there is no risk of confusion we abuse notation and identify a structure with its underlying universe; for instance, we write $a\in A$ to mean that $a$ is an element of the universe of $A$, we write $|A|$ for the cardinality of that universe, and if $G$ extends $F$ by one element we write $G=F\cup\{w\}$ for the appropriate $w$.} By $\age(X)$ we denote the set of all finite substructures of $X$.

We write $A \simeq B$ to indicate that $A$ and $B$ are isomorphic, $A \leq B$ to indicate that $A$ is a substructure of $B$ (for all $a \in [A]^{<\omega}$ and every $R \in \mathcal{L}$, $a \in R(A)$ if and only if $a \in R(B)$; in the class of finite graphs, \textit{subgraph} means \textit{induced subgraph}), $A<B$ for proper substructure, and $A \embed B$ for embeddability.

Recall that $\mathcal{F}$ is a \textit{Fra\"iss\'e class} if the following hold:
\begin{enumerate}
    \item \textit{Joint Embedding Property} (JEP): if $A,B\in \mathcal{F}$, then there exists $C\in \mathcal{F}$ such that $A,B\embed C$;
    \item \textit{Amalgamation Property} (AP): if $A,B,C \in \mathcal{F}$ with $A\embed B$ and $A\embed C$, then there exists $D\in \mathcal{F}$ such that $B,C \embed D$ and the respective diagram commutes.
\end{enumerate}
In this case, there exists a unique (up to isomorphism) countable $X \in \sigma\mathcal{F}$ that is both universal ($\age(X)=\cF$) and ultrahomogeneous (every isomorphism between finite substructures extends to an automorphism). This $X$ is the \textit{Fra\"iss\'e limit} of $\mathcal{F}$~\cite{fraisse, hodges-model-theory}.

\begin{mydef}
Let $A,B$, and $C$ be structures. We say that $B$ and $C$ are \textit{isomorphic over} $A$ if $A\leq B, C$, and there is an isomorphism from $B$ to $C$ that is the identity on $A$.
\end{mydef}

\begin{mydef}
Let $A,B,C$, and $X$ be structures such that $A\leq B$ and $A\leq X$. Then:
    \begin{itemize}
        \item We say that $B$ is a \textit{prime extension} of $A$ if $|B\setminus A|=1$.
        \item If $C\leq X$, then we say that $C$ is \textit{a realization of $B$ in $X$} if $B$ and $C$ are isomorphic over $A$.
    \end{itemize}
\end{mydef}

\begin{mydef}\label{2:definition_rankFunction}
Let $X\in\sigma\cF$, $F\in\age(X)$ and $\alpha$ an ordinal.
    \begin{itemize}
        \item $\rk_X(F)\geq 0$ always.
        \item $\rk_X(F)\geq \alpha+1$ if every prime extension
              of $F$ in $\mathcal{F}$ has a realization $C$ in $X$ such that 
              $\rk_X(C)\geq\alpha$.
        \item For $\alpha$ a limit ordinal, $\rk_X(F)\geq\alpha$ 
              if $\rk_X(F)\geq \beta$ for every $\beta <\alpha$.
        \item $\rk_X(F)=\sup\defset{\alpha}{\rk_X(F)\geq\alpha}$, 
              with the convention that $\rk_X(F)=\infty$ if 
              $\rk_X(F)\geq\alpha$ for every 
              ordinal $\alpha$.\footnote{We adopt the convention 
              that $\infty>\alpha$ for every ordinal $\alpha$.}
        \item $\rk(X):=\rk_X(\emptyset)$.
    \end{itemize}
\end{mydef}

A routine transfinite induction shows that if $\rk_X(F)\geq\alpha$ and $\beta\leq\alpha$, then $\rk_X(F)\geq\beta$. Also, it is standard to check that the rank function takes values only on $\w_1\cup\{\infty\}$.

\begin{remark}\label{characterization of rank 0}
For $F\in\age(X)$, we have $\rk_X(F)=0$ if and only if there is some prime extension $G$ of $F$ that cannot be realized in $X$.
\end{remark}

Note that Remark~\ref{characterization of rank 0} applied to the class of graphs shows that for a graph $X$ and a vertex $v \in X$, we have $\rk_X(\{v\}) = 0$ if and only if $v$ is an isolated vertex or $v \sim x$ for all $x \in X \setminus \{v\}$. This is because in both scenarios, there exists a prime extension of $\{v\}$ that cannot be realized in $X$: in the first case, the prime extension is $K_2$, and in the second, it is the complement of $K_2$.

\begin{mydef}\label{def rank game}
Let $X\in\sigma\cF$ and $F\in\age(X)$. The \emph{rank game on $X$ starting from $F$} is played by two players, \textsf{I} and \textsf{II}, over $\w$ rounds. At round $n$ (with current position $F_n\in\age(X)$, where $F_0=F$):
\begin{enumerate}
    \item \textsf{I} chooses a prime extension $B$ of $F_n$ in $\cF$.
    \item \textsf{II} responds with a realization $C$ of $B$ in $X$, and sets $F_{n+1}=C$.
\end{enumerate}
\textsf{II} loses at round $n$ if no realization exists; \textsf{II} wins if the game lasts all $\w$ rounds.
\end{mydef}

\begin{proposition}\label{rank via game}
Let $X\in\sigma\cF$ and $F\in\age(X)$. Then $\rk_X(F)=\infty$ if and only if \textsf{II} has a winning strategy in the rank game on $X$ starting from $F$.
\end{proposition}
\begin{proof}
$(\Rightarrow)$: If $\rk_X(F)=\infty$, then every prime extension of $F$ has a realization $C$ in $X$ with $\rk_X(C)=\infty$. This gives \textsf{II} a strategy: at each round, respond with a realization of infinite rank. Since infinite rank is preserved at every step, the game never terminates.

$(\Leftarrow)$: Suppose $\rk_X(F)=\alpha<\infty$. Since $\rk_X(F)\not\geq\alpha+1$, some prime extension $B$ of $F$ has all realizations in $X$ of rank less than $\alpha$. \textsf{I} plays $B$, forcing $\rk_X(F_1)<\alpha$. Iterating, \textsf{I} produces a strictly decreasing sequence of ordinals, which must terminate.
\end{proof}

The rank game can be viewed as an asymmetric variant of the abstract Banach--Mazur game studied by Krawczyk and Kubi\'s~\cite{krawczyk-kubis-games} in the context of evolution systems~\cite{kubis-radecka, kubis-fraisse-sequences}. In that game, both players alternately choose concrete finite substructures from $\age(X)$ to form an increasing chain.

In contrast, the rank game is fundamentally asymmetric: Player~\textsf{I} proposes \emph{abstract} prime extensions from the class $\cF$, while Player~\textsf{II} must find \emph{concrete realizations} of those extensions inside the fixed structure $X$. The rank measures how many rounds \textsf{II} can survive this challenge.

It might seem that if $F, G \in \age(X)$ are isomorphic, then we would have $\rk_X(F) = \rk_X(G)$. However, this is clearly false; for example, if $X$ is a graph and $v, w \in X$ where $v$ is isolated and $w$ is not isolated nor connected to every other vertex, then $\rk_X(\{v\}) = 0 < \rk_X(\{w\})$, while $\{v\} \simeq \{w\}$.

On the other hand, if $h:X\to X$ is an automorphism of $X$ such that $h[F]=G$, then $\rk_X(F)=\rk_X(G)$. This follows by a straightforward induction: $h$ maps prime extensions of $F$ to prime extensions of $G$ and realizations in $X$ to realizations in $X$, preserving the rank at each level.

\begin{proposition}\label{rank of substructure geq than the rank of its prime extensions}
Assume $\cF$ has the amalgamation property. Let $X\in\sigma\cF$ and $\alpha$ an ordinal. If $F,G\in\age(X)$ are such that $G$ is a prime extension of $F$ and $\rk_X(G)\geq\alpha$, then $\rk_X(F)\geq\alpha$. In particular, $\rk_X(F)\geq\rk_X(G)$.
\end{proposition}
\begin{proof}
By induction on $\alpha$. The case $\alpha=0$ is straightforward, as everything has rank at least $0$.

\textbf{Successor case:} Assume the result holds for $\alpha$ and let us prove it for $\alpha+1$. Let $F,G\in\age(X)$ with $G=F\cup\{w\}$ a prime extension of $F$ and $\rk_X(G)\geq\alpha+1$. Let $B=F\cup\{y\}\in\cF$ be any prime extension of $F$; we seek a realization $C\in\age(X)$ of $B$ with $\rk_X(C)\geq\alpha$. By AP applied to $G$ and $B$ over $F$, there exists $H\in\cF$ with $G\leq H$ and an embedding $\iota:B\to H$ that is the identity on $F$; set $y':=\iota(y)\in H\setminus F$ (since $\iota$ is injective and the identity on $F$).
\\
\textit{Subcase 1: $y'\in G$.} Then $y'=w$ and $\iota(B)=G$, thus $G$ itself is a realization of $B$ in $X$ with $\rk_X(G)\geq\alpha+1\geq\alpha$.
\\
\textit{Subcase 2: $y'\in H\setminus G$.} Replacing $H$ by its induced substructure on $G\cup\{y'\}$ ---which lies in $\cF$ since $\cF$ is hereditary and still contains both $G$ and the image of $\iota$--- we may assume $H=G\cup\{y'\}$, a prime extension of $G$ in $\cF$. Since $\rk_X(G)\geq\alpha+1$, there is a realization $D\in\age(X)$ of $H$ with $\rk_X(D)\geq\alpha$ and an isomorphism $f:H\to D$ with $f\rest G=\text{id}_G$. Setting $z:=f(y')$, the composition $f\circ\iota:B\to D$ is an embedding which is the identity on $F$ and has image $F\cup\{z\}$, then $C:=F\cup\{z\}$ is a realization of $B$ in $X$. Since $D=C\cup\{w\}$ is a prime extension of $C$ in $\age(X)$ with $\rk_X(D)\geq\alpha$, the inductive hypothesis gives $\rk_X(C)\geq\alpha$.

\textbf{Limit case:} Suppose that $\beta$ is a limit ordinal, and the result holds for all $\alpha<\beta$. Let $F,G\in\age(X)$ be such that $G$ is a prime extension of $F$ and $\rk_X(G)\geq\beta$. In particular, $\rk_X(G)\geq\alpha$ for every $\alpha<\beta$. By the inductive hypothesis, $\rk_X(F)\geq\alpha$ for all $\alpha<\beta$, which implies $\rk_X(F)\geq\beta$.
\end{proof}

\begin{corollary}\label{rank function of a substructure is geq}
Assume $\cF$ has AP. If $F,G\in\age(X)$ are such that $F\leq G$, then $\rk_X(F)\geq\rk_X(G)$.
\end{corollary}
\begin{proof}
We can construct a chain $(F_i)_{i\in n}$ of structures in $\age(X)$ such that $F_{i+1}$ is a prime extension of $F_i$ for all $i< n-1$, with $F_0=F$ and $F_{n-1}=G$. Therefore, by Proposition~\ref{rank of substructure geq than the rank of its prime extensions}, we obtain:
    $$\rk_X(F)=\rk_X(F_0)\geq\rk_X(F_1)\geq\dots\geq\rk_X(F_{n-1})=\rk_X(G).$$
\end{proof}

Applying Corollary~\ref{rank function of a substructure is geq} to the case $F=\emptyset$, we deduce the following:

\begin{corollary}\label{1.11}
Assume $\cF$ has AP. For every $X\in\sigma\cF$, we have $\rk(X)\geq\rk_X(G)$ for all $G\in\age(X)$.
\end{corollary}

The amalgamation hypothesis cannot be dropped. Consider the class $\cF$ of finite sets with a unary predicate $u$ satisfying the axiom $(\exists x_0,x_1)(x_0\neq x_1\wedge u(x_0)\wedge u(x_1))\to(\forall y)(u(y))$. Let $X$ be an infinite set with $u(x)$ for all $x\in X$. Then for any $x_0,x_1\in X$, the extension $\{x_0,y\}$ with $\neg u(y)$ is in $\cF$ but unrealizable in $X$, giving $\rk_X(\{x_0\})=0$. On the other hand, every extension of $\{x_0,x_1\}$ in $\cF$ must satisfy $u(y)$, and all such extensions are realizable in $X$; iterating gives $\rk_X(\{x_0,x_1\})=\infty$.

\begin{lemma}\label{rank on substructures vs rank on the total}
Let $X \in \sigma\cF$, $Y \leq X$, $F \in \age(Y)$ and $\alpha \in \w_1$. If $\rk_Y(F)\geq \alpha$, then $\rk_X(F) \geq \alpha$. In particular, $\rk_Y(F) \leq \rk_X(F)$.
\end{lemma}
\begin{proof}
We proceed by induction on $\alpha$. If $\alpha = 0$, the result is immediate. Now assume that the result holds for all $\beta < \alpha$, and let $F$ be such that $\rk_Y(F)\geq \alpha$.

If $\alpha = \beta + 1$, then since $\rk_Y(F) \geq \beta + 1$, for every prime extension $H$ of $F$, there exists a realization $B \in \age(Y)$ of $H$ with $\rk_Y(B) \geq \beta$. By the inductive hypothesis, $\rk_X(B) \geq \beta$, which establishes that $\rk_X(F) \geq \beta + 1 = \alpha$.

If $\alpha$ is a limit ordinal, then since $\rk_Y(F) \geq \beta$ for every $\beta < \alpha$, the inductive hypothesis implies $\rk_X(F) \geq \beta$ for all $\beta < \alpha$, giving $\rk_X(F) \geq \alpha$.
\end{proof}

\begin{corollary}\label{1.13}
If $Y\leq X$, then $\rk(Y)\leq\rk(X)$.
\end{corollary}

\begin{proposition}\label{intermediate values}
Let $X\in \sigma \mathcal{F}$ and $F\in \age(X)$ such that $\rk_X (F)=\alpha$ for some ordinal $\alpha<\infty$. If $\beta <\alpha$ then there is $E\in \age(X)$ such that $\rk_X (E)=\beta$.
\end{proposition}

\begin{proof}
By induction on $\alpha$. The case $\alpha=0$ is vacuous. 

\textit{Case $\alpha=\gamma+1$:} it suffices to find $E\in\age(X)$ with $\rk_X(E)=\gamma$, then apply the induction hypothesis if $\beta<\gamma$. Since $\rk_X(F)\geq\gamma+1$, every prime extension of $F$ has a realization in $X$ with rank at least $\gamma$. Since $\rk_X(F)\not\geq\gamma+2$, there exists a prime extension $B$ such that every realization of $B$ has rank at most  $\gamma$. For this $B$, combining both facts gives a realization $E$ with $\rk_X(E)=\gamma$.

\textit{Case $\alpha$ limit:} since $\beta<\alpha$ and $\alpha$ is a limit, $\beta+2<\alpha$, then $\rk_X(F)>\beta+2$. Hence, every prime extension of $F$ has a realization with rank at least $\beta+1$. On the other hand, since $\rk_X(F)\not\geq\alpha+1$, there exists a prime extension $B_0$ such that every realization of $B_0$ in $X$ has rank strictly below $\alpha$. For this $B_0$, the two facts combine: there exists a realization $E$ with $\beta<\rk_X(E)<\alpha$. Apply the induction hypothesis to $E$ and $\rk_X(E)$ to obtain a structure of rank exactly $\beta$.
\end{proof}

\section{Rank and the amalgamation property}\label{sec rank and AP}

We now explore the relationship between the rank function and the amalgamation property. Assume $\cF$ as in the previous section. The main results are: the rank controls embeddability of finite structures, $\rk(X)\geq\w$ characterizes Fra\"iss\'e classes and full ages, $\rk(X)=\infty$ characterizes embeddability of the Fra\"iss\'e limit, and the Rank Property implies that $\cF$ is Fra\"iss\'e. The converse does not hold: being Fra\"iss\'e does not suffice for RP, as the example of finite sets shows.

Let $A, A',B$ be structures such that $A\leq B$ and $A$ is isomorphic to $A'$. Denote $B(A,A')$ to the structure isomorphic to $B$ resulted in changing the elements of $A$ in $B$ by the respective elements in $A'$.

\begin{lemma}\label{lemma:1}
Let $X\in \sigma\mathcal{F}$, $A,B\in \mathcal{F}$ and $l\in \omega$, such that $\rk(X)\geq l$ and there is an embedding from $A$ to $B$. Assume that $|A|\leq |B|\leq l$ and there is some $A'\in \age(X)$ isomorphic to $A$ with $\rk_X(A')\geq l-|A|$. Then, there exists $B'$ realization of $B(A,A')$ in $X$ extending $A'$ such that $\rk_X(B')\geq l-|B|$.
\end{lemma}

\begin{proof}
We prove it by induction on $n-m$, where $n=|B|$ and $m=|A|$. Without loss of generality $A\subseteq B$. If $n-m=0$, $A'$ witnesses the statement. Consider $n-m=k+1$. Let $x\in B\setminus A$ and $C=A\cup \set{x}$. By definition of the rank function, we find $C'$ a realization of $C(A,A')$ in $X$ extending $A'$ such that $\rk_X(C')\geq l-(m+1)$. Now apply the inductive hypothesis for $C'$ and $B$ to find $B'$ realization of $B(C,C')$ extending $C'$ (and hence $A'$) such that $\rk_X(B')\geq l-n$.
\end{proof}

Applying Lemma \ref{lemma:1} with $A= \emptyset$ we get the following:

\begin{corollary}\label{rk geq n embeds all}
Let $X\in\sigma\cF$ with $\rk(X)\geq l$. Then every $B\in\cF$ with $|B|\leq l$ can be embedded into $X$.
\end{corollary}

The converse of Corollary~\ref{rk geq n embeds all} does not hold in general. In the class of finite graphs, the converse holds for $n\leq 2$ but fails for $n=3$: the graph on $5$ vertices $\{1,2,3,4,5\}$ with edges $1$-$2$, $1$-$3$, $2$-$3$, $1$-$4$ contains all graphs on $3$ vertices as induced subgraphs, yet has rank $2$. Moreover, $5$ vertices is optimal: in any graph on $4$ vertices containing $K_3$, three vertices form a triangle, and then every triple includes at least two of them, then $\overline{K_3}$ cannot be embedded. In the class of finite tournaments (see Section~\ref{tournaments section}), the failure already occurs at $n=2$: the unique tournament on $2$ vertices contains every tournament of size $\leq 2$, but has rank $1$. In the class of finite linear orders, the failure is even more dramatic: any finite linear order of size $m$ contains all linear orders of size $\leq m$, yet its rank is only $\lfloor\log_2(m+1)\rfloor$ (Theorem~\ref{rank of finite linear orders}).

In contrast, for classes defined by unary predicates alone (with no relations of arity $\geq 2$), the converse of Corollary~\ref{rk geq n embeds all} does hold. If $\cF$ is the class of finite structures in a language of $k$ unary predicates, then $\rk_X(F)\geq n$ if and only if each of the $2^k$ color patterns has at least $n$ representatives in $X\setminus F$, which is also the condition for embedding all structures of size $\leq n$. These classes are Fra\"iss\'e but fail RP: every infinite $X$ with $\age(X)=\cF$ has $\rk(X)=\infty$, then the image of $\rk$ on $\sigma\cF$ is $\w\cup\{\infty\}$. This suggests that relations of arity $\geq 2$ are essential for the Rank Property.

We now turn to the connection between rank and the amalgamation property. The key observation is that sufficiently high rank forces $\cF$ to be Fra\"iss\'e.

\begin{theorem}\label{rk geq omega implies AP}
There is $X\in\sigma\cF$ such that $\rk(X)\geq\w$ if and only if $\cF$ is a Fra\"iss\'e class.
\end{theorem}
\begin{proof}
$(\Leftarrow)$: If $\cF$ is a Fra\"iss\'e class, the Fra\"iss\'e limit has rank $\infty\geq\w$.

$(\Rightarrow)$: Let $X\in\sigma\cF$ with $\rk(X)\geq\w$. For AP: given $A,B,C\in\cF$ with $A\embed B$ and $A\embed C$, apply Lemma~\ref{lemma:1} to $\emptyset$ and $A$ to obtain $A'\in\age(X)$ isomorphic to $A$ with $\rk_X(A')\geq l-|A|$, where $l=\max\{|B|,|C|\}+1$; then apply Lemma~\ref{lemma:1} to $A',B$ and $A',C$ respectively, obtaining $B',C'\in\age(X)$ realizations of $B(A,A')$ and $C(A,A')$ extending $A'$. Then $B'\cup C'\in\age(X)\subseteq\cF$ witnesses AP. For JEP: since $\cL$ has no constants and all relations have arity $\geq 1$, the empty structure is an $\cL$-structure and lies in $\cF$ by heredity; applying AP to $\emptyset\embed A,B$ for any $A,B\in\cF$ yields a common extension.
\end{proof}

\begin{corollary}\label{RP implies Fraisse}
If $\cF$ has the Rank Property, then $\cF$ is a Fra\"iss\'e class.
\end{corollary}
\begin{proof}
By RP, there exists $X\in\sigma\cF$ with $\rk(X)\geq\w$, then $\cF$ is Fra\"iss\'e by Theorem~\ref{rk geq omega implies AP}.
\end{proof}

The converse of Corollary~\ref{RP implies Fraisse} does not hold. As observed after Corollary~\ref{rk geq n embeds all}, classes defined by unary predicates alone are Fra\"iss\'e but fail RP; the simplest instance is the class of finite sets with the empty language.

\begin{proposition}\label{rk geq omega iff age}
Suppose that for every $n\in\w$ there exists $H\in\cF$ with $\rk(H)\geq n$.\footnote{This hypothesis is verified for each class studied in Sections~\ref{FAP and FEP section}--\ref{tournaments section}.} Then for every $X\in\sigma\cF$,
$$\rk(X)\geq\w\quad\text{if and only if}\quad\age(X)=\cF.$$
\end{proposition}
\begin{proof}
$(\Rightarrow)$: By Corollary~\ref{rk geq n embeds all}, $\rk(X)\geq n$ implies that every $A\in\cF$ with $|A|\leq n$ embeds into $X$. Since this holds for all $n$, $\age(X)=\cF$.

$(\Leftarrow)$: By hypothesis, for each $n$, there exists $H_n\in\cF$ with $\rk(H_n)\geq n$. Since $\age(X)=\cF$, we have $H_n\in\age(X)$; therefore, $H_n\leq X$ and $\rk(X)\geq\rk(H_n)\geq n$ by Corollary~\ref{1.13}. Since this holds for all $n$, $\rk(X)\geq\w$.
\end{proof}

Finally, as mentioned in the introduction, the rank function characterizes when a countable structure contains the Fra\"iss\'e limit. We include the short proof for completeness.

\begin{proposition}[Kubi\'s--Shelah {\cite{kubis-shelah}}]\label{rk infty iff universal}
For every $X\in\sigma\cF$, $\rk(X)=\infty$ if and only if $\cF$ is a Fra\"iss\'e class and its Fra\"iss\'e limit embeds in $X$.
\end{proposition}
\begin{proof}
$(\Leftarrow)$: If $\cF$ is Fra\"iss\'e with limit $U$ and $U\embed X$, then $\rk(X)\geq\rk(U)$ by Corollary~\ref{1.13}; and $\rk(U)=\infty$ since by ultrahomogeneity every prime extension of every $F\in\age(U)$ is realized in $U$.

$(\Rightarrow)$: If $\rk(X)=\infty$, then $\rk(X)\geq\w$, thus by Theorem~\ref{rk geq omega implies AP} $\cF$ is a Fra\"iss\'e class; let $U$ be its limit. Observe that, for $F\in age(X)$,  $\rk_X(F)=\infty$ if and only if every prime extension $B$ of $F$ has a realization $B'$ in $X$ with $\rk_X(B')=\infty$. This allows us to construct a chain $\emptyset=F_0\leq F_1\leq F_2\leq\cdots$ in $X$ with $\rk_X(F_n)=\infty$ for all $n$, choosing at each step a realization of any prime extension is needed to build $U$. Therefore, $\bigcup_{n<\w}F_n$ is isomorphic to $U$.
\end{proof}

The results of this section form a hierarchy of universality controlled by the rank: $\rk(X)\geq n$ implies that every structure in $\cF$ of size at most $n$ embeds into $X$ (Corollary~\ref{rk geq n embeds all}); $\rk(X)\geq\w$ is equivalent to $\age(X)=\cF$ and forces $\cF$ to be Fra\"iss\'e (Theorem~\ref{rk geq omega implies AP}, Proposition~\ref{rk geq omega iff age}); and $\rk(X)=\infty$ is equivalent to $\cF$ being Fra\"iss\'e and its limit embedding into $X$ (Proposition~\ref{rk infty iff universal}). The countable ordinal values between $\w$ and $\infty$ thus provide a fine measure of how far a structure is from containing the Fra\"iss\'e limit.

The amalgamation property enters Section~\ref{sec rank} only once: in the proof of Proposition~\ref{rank of substructure geq than the rank of its prime extensions}, which amalgamates an abstract prime extension $B$ of $F$ with the realized prime extension $G$ to obtain a common extension in $\cF$. Once that proposition is established, no later proof reconstructs this kind of amalgamation step: each of the three settings developed in the remaining sections encodes prime extensions concretely --- $\cF$-good types under FEP in Section~\ref{FAP and FEP section}, subsets $S\subseteq F$ for tournaments in Section~\ref{tournaments section}, and intervals induced by $F$ in $Y$ for linear orders in Section~\ref{sec linear orders} --- then any extension exhibited in a proof already belongs to $\cF$ by construction, and Proposition~\ref{rank of substructure geq than the rank of its prime extensions} is invoked only as a black box.

\section{The Rank Property under FAP and FEP}\label{FAP and FEP section}

In this section we prove the Rank Property for classes satisfying the free amalgamation property and the full extension property. The definitions and constructions are stated for a finite arbitrary relational language, but the reader may find it helpful to keep in mind the class of finite graphs as the guiding example throughout: there, a good type of $G$ (see definitions \ref{def de type} and \ref{def de tipos de tipos}) is simply a choice of neighborhood $A\subseteq G$ for the new vertex, and the full extension adds a vertex adjacent to every element of $G$. The general setup below abstracts this picture to accommodate higher-arity relations and unary predicates.

Let $\mathcal L$ be a language with finitely many (at least one) relational symbols, each of finite arity. More precisely, write
\[
\mathcal L=(R_i,S_j \mid i\in N,\; j\in M),
\]
where $N$ and $M$ are finite index sets, every $S_j$ is a unary relation symbol (arity $1$), and $c_i\geq 2$ denotes the arity of $R_i$ for each $i\in N$. We assume that $\mathcal{F}$, the class of finite $\mathcal{L}$-structures, satisfies the free amalgamation property and the irreflexivity axiom: for every $i\in N$, if $(x_0,\dots, x_m)\in R_i$ and $l\neq n$, then $x_l\neq x_n$.

\begin{mydef}\label{def de type}
Let $G\in\cF$. A \emph{type of $G$} is a pair $T=(A,s)$ such that:
\begin{enumerate}
  \item $A=(A_i^j \mid i\in N,\ j<c_i)$, where for every $i\in N$ and every $j<c_i$ we have $A_i^j\subseteq G^{c_i-1}$, and
  \item $s\in 2^M$.
\end{enumerate}
Given a type $T=(A,s)$ of $G$, the \emph{prime extension of $G$ encoded by $T$} is the (unique up to isomorphism) prime extension $H=G\cup\{z\}$ characterized by:
\begin{enumerate}
  \item For every $i\in N$, every $j<c_i$, and every tuple $(b_0,\dots,b_{c_i-2})\in G^{c_i-1}$,
  \[
    (b_0,\dots,b_{j-1},z,b_j,\dots,b_{c_i-2})\in R_i(H)
    \quad\Longleftrightarrow\quad
    (b_0,\dots,b_{c_i-2})\in A_i^j.
  \]
  \item For each $j\in M$, we have $z\in S_j$ if and only if $s(j)=1$.
\end{enumerate}
We denote this prime extension by $G_T$.
\end{mydef}

\begin{mydef}\label{def de tipos de tipos}
\begin{enumerate}
    \item A type $T$ of $G$ is \textit{empty} if $A_i^j=\emptyset$ for all $(i,j)\in N\times c_i$.
    \item A type $T$ of $G$ is \emph{full} if for every $x\in G$, there are $i\in N$ and $j\in c_i$ such that $x\in A^j_i$.
    \item A type $T$ of $G$ is $\cF$\textit{-good} if $G_T\in\cF$.
    \item A prime extension $F$ of $G$ is a \textit{full prime extension} of $G$ if there is a full type $T$ of $G$ that is $\mathcal{F}$-good and $F$ is isomorphic to $G_T$.
\end{enumerate}
\end{mydef}

Note that if $T$ is an empty type, then $G_T$ is a copy of $G$ with a new isolated vertex. In the class of graphs, the unique full $\cF$-good type of $G$ adds a vertex adjacent to every element of $G$.

\begin{mydef}\label{FEP def}
We say that $\cF$ satisfies the \textit{full extension property} (FEP) if every $G\in\cF$ has a full $\cF$-good type.
\end{mydef}

Note that FEP is a purely combinatorial condition on the finite class $\cF$: it asserts that a certain finite extension always belongs to $\cF$. Combined with FAP, it provides the combinatorial tools needed to control the rank from above. This reflects the approach of the present paper, where the Rank Property is established by finite combinatorial constructions rather than model-theoretic methods.

\begin{mydef}\label{def de adjacencia}
Let $G\in\sigma\cF$.
\begin{enumerate}
    \item If $x,y\in G$ and $i\in N$, we say that \emph{$x$ is $i$-adjacent to $y$ in $G$} if there is a $c_i$-tuple $(b_0,\dots,b_{c_i-1})\in R_i(G)$ such that $x=b_l$ and $y=b_m$ for some $l,m\in c_i$ with $l\neq m$. We denote this by $x\simup{G}_iy$.\footnote{Note that $x\simup{G}_i y$ if and only if $y\simup{G}_i x$.}
    \item We say that $x$ and $y$ are \emph{adjacent in $G$}, denoted by $x\simup{G} y$, if $x\simup{G}_i y$ for some $i\in N$.
    \item If $H\subseteq G$, then $H$ is \textit{complete in $G$} if $x\simup{G}y$ for all $x,y\in H$ with $x\neq y$.
    \item If $A,B\subseteq G$ are disjoint, we say that \emph{$A$ is disconnected from $B$ in $G$} if $x\not\simup{G} y$ for all $x\in A$ and all $y\in B$, and write $A\not\simup{G}B$.
\end{enumerate}
\end{mydef}

From now on, in this section, assume that $\cF$ satisfies FEP.

\begin{lemma}\label{continue with completes caso general}
Let $X\in\sigma\cF$ and $F\in\age(X)$ be such that $\rk_X(F)\geq\gamma+m$, where $\gamma$ is a limit ordinal or $0$. Then there is $G=\{g_l\mid l\in m\}\in\age(X\setminus F)$ such that $G$ is complete in $X$, $\rk_X(F\cup G)\geq \gamma$, and $v\simup{X}w$ for all $v\in F$ and all $w\in G$. In particular, if $\rk(X)\geq\gamma+m$, there exists $G=\{g_l\mid l\in m\}\in\age(X)$ that is complete in $X$ and $\rk_X(G)\geq\gamma$.
\end{lemma}
\begin{proof}
As $\cF$ satisfies FEP, we construct recursively $(F_k)_{k\leq m}\subseteq \age(X)$ such that $F_0=F$ and for all $k<m$: (1) $F_{k+1}$ is a full prime extension of $F_k$, and (2) $\rk_X(F_k)\geq \gamma +(m-k)$. Then $G:=F_m\setminus F$ is as desired.
\end{proof}
 
The first step is to construct, for each $n\geq 1$, a finite structure $H_n$ of exact rank $n$. These serve as universal structures at each finite level: $H_n$ realizes every prime extension of every transversal of its levels, which drives the lower bound on the rank.

We recursively construct a sequence $(X_n)_{n \in \w}$ of finite sets and a sequence of $\cL$-structures $(H_n)_{n\in\w}$ such that:
\begin{enumerate}
    \item The domain of $H_n$ is $\bigcup_{i\in n}X_i$.
    \item $X_0=\{v_T\mid T\text{ is an }\cF\text{-good type of }\emptyset\}$, with $v_T\neq v_{T'}$ for $T\neq T'$. Each $v_T$ realizes $T$ in $H_n$ for every $n\geq 1$.
    \item $X_k \cap X_l = \emptyset$ if $k \neq l$.
    \item If $x,y\in X_k$ with $x\neq y$, then $x\not\simup{H_l}y$ for all $l\geq k$.
    \item For every $j \in\w$, each $\vec{x} = (x_i)_{i \in j} \in \prod_{i \in j} X_i$, and any $\cF$-good type $T$ of $\{x_i\mid i\in j\}$, there exists a vertex $v(\vec{x}, T) \in X_j$ such that $\{x_i\mid i\in j\}\cup\{v(\vec{x}, T)\}$ is the prime extension of $\{x_i\mid i\in j\}$ encoded by $T$. Also, if $(\vec{x},T_0)\not= (\vec{y},T_1)$ then $v(\vec{x},T_0)\not= v(\vec{y},T_1)$
\end{enumerate}

The construction proceeds recursively, starting from $H_0=\emptyset$. Suppose $X_0,\ldots,X_{j-1}$ and $H_j$ have been defined for some $j\geq 0$ (so $H_j$ has domain $\bigcup_{i\in j}X_i$, with $H_0=\emptyset$). For every $\vec{x}=(x_i)_{i\in j}\in\prod_{i\in j}X_i$ and every $\cF$-good type $T$ of $\{x_i\mid i\in j\}$, let $v(\vec{x},T)$ be a new vertex with $v(\vec{x},T_0)\neq v(\vec{y},T_1)$ whenever $(\vec{x},T_0)\neq(\vec{y},T_1)$. Set
\[
X_j=\{v(\vec{x},T)\mid(\vec{x}\in\prod_{i\in j}X_i)\wedge(T\text{ is a }\cF\text{-good type of }\{x_i\mid i\in j\})\}.
\]
For every such $(\vec{x},T)$, freely amalgamate $H_j$ and $\{x_i\mid i\in j\}\cup\{v(\vec{x},T)\}$ over $\{x_i\mid i\in j\}$. Then freely amalgamate the resulting collection to obtain $H_{j+1}$. At $j=0$, the empty product yields $\vec{x}=\emptyset$ and $\{x_i\mid i\in 0\}=\emptyset$, thus the amalgamations are taken with $H_0=\emptyset$ over $\emptyset$, producing $X_0$ (matching (2)) and $H_1$.

\begin{lemma}\label{rank on the Hn graph caso general}
Let $n\geq 1$. Then for all $j \leq n$ and all $(x_i)_{i \in j} \in \prod_{i \in j} X_i$, we have $\rk_{H_n}(\{x_i \mid i \in j\}) \geq n - j$.
\end{lemma}
\begin{proof}
Fixing $n$, we proceed by induction on $l = n - j$. If $n - j = 0$, the result holds. Suppose the result is true for $l = n - j$ and let us prove it for $l + 1 = n - (j - 1)$. Let $\vec{x} = (x_i)_{i \in j - 1} \in \prod_{i \in j - 1} X_i$. Every prime extension of $\{x_i \mid i \in j - 1\}$ in $\cF$ is encoded by an $\cF$-good type $T$, and is realized in $H_n$ by $\{x_i \mid i \in j - 1\} \cup \{v(\vec{x}, T)\}$. By the induction hypothesis, $\rk_{H_n}(\{x_i \mid i \in j - 1\} \cup \{v(\vec{x}, T)\}) \geq n - j$, hence $\rk_{H_n}(\{x_i \mid i \in j - 1\}) \geq n - (j - 1)$.
\end{proof}

Note that applying Lemma~\ref{rank on the Hn graph caso general} with $j=1$, every prime extension of $\emptyset$ has a realization with rank at least $n-1$ in $H_n$, thus $\rk(H_n)\geq n$.

\begin{lemma}\label{Hn does not contain big completes caso general}
Let $n\geq 1$ and $G\subseteq H_n$ such that $|G|= n+1$. Then $G$ is not complete in $H_n$.
\end{lemma}
\begin{proof}
Let $G= \{w_j\mid j\in n+1\}$. Since $H_n = \bigcup_{i \in n} X_i$, there must exist distinct $j, j' \in n+1$ such that $w_j, w_{j'} \in X_i$ for some $i \in n$, and thus $w_j\not\simup{H_n}w_{j'}$.
\end{proof}

\begin{corollary}\label{exact rank of Hn caso general}
For every $n\geq 1$ we have $\rk(H_n)=n$.
\end{corollary}
\begin{proof}
We already know $\rk(H_n) \geq n$. If $\rk(H_n) \geq n + 1$, by Lemma~\ref{continue with completes caso general}, there is $G\in\age(H_n)$ with $|G|=n+1$ and $G$ complete in $H_n$, contradicting Lemma~\ref{Hn does not contain big completes caso general}.
\end{proof}

By Lemma~\ref{rank on the Hn graph caso general}, $\rk_{H_n}(\{v\}) \geq n - 1$ for all $n\geq 1$ and all $v\in X_0$.

Note that for every $G\in\cF$ of size at most $n$ and any enumeration $\{y_i\}_{i \in m}$ of its vertices, there exists an embedding $f: G \to H_n$ such that $f(y_i) \in X_i$ for all $i \in m$.

\begin{theorem}\label{sufficient condition for having rank at least n caso general}
For every $n \geq 1$ and every $X\in\sigma\cF$, if $H_n$ can be embedded into $X$, then $\rk(X) \geq n$.
\end{theorem}
\begin{proof}
Assume $H_n \leq X$. By Lemma~\ref{rank on the Hn graph caso general}, $\rk_{H_n}(\{v\}) \geq n - 1$ for all $v\in X_0$, and by Lemma~\ref{rank on substructures vs rank on the total}, $\rk_{X}(\{v\}) \geq n - 1$, which implies $\rk_X(\emptyset) \geq n$.
\end{proof}

\begin{corollary}
For every $n \geq 1$ and $X\in \sigma\mathcal{F}$, there exists some $N$ such that if $X$ contains all structures of cardinality $N$, then $\rk(X) \geq n$.
\end{corollary}

\begin{corollary}\label{if G contains all finite graphs then has rank at least omega caso general}
Let $X\in \sigma \mathcal{F}$. If every element in $\mathcal{F}$ can be embedded into $X$, then $\rk(X) \geq \omega$.
\end{corollary}

With the finite universal structures in hand, we now build countable structures of arbitrary prescribed rank. The key device is that of a \emph{kernel}: a finite substructure $H$ over which infinitely many ``leaves'' are amalgamated, with FAP ensuring that distinct leaves are pairwise disconnected. This disconnection is what controls the upper bound on the rank.

\begin{mydef}
Let $H \in \cF$ and $X \in \sigma\cF$ be such that $H \leq X$. We say that $H$ is a \textit{kernel of $X$} if there exists a family $\{Y_n \mid n \in \omega\}$ such that $\emptyset \neq Y_n \leq X$ for all $n \in \omega$, $X = (\bigsqcup_{n \in \omega} Y_n) \sqcup H$, and $Y_n\not\simup{X}Y_m$ for $n\neq m$.
\end{mydef}

\begin{lemma}\label{rank bounded by cardinality of kernel caso general}
Let $X\in \sigma\mathcal{F}$ and $H< X$ with $H$ a kernel of $X$, and let $Y\leq X\setminus H$ satisfy $w\not\simup{X} Y$ for every $w\in X\setminus (Y\cup H)$. Then for all $F\in\age(Y\cup H)$ with $F\cap Y\neq\emptyset$ and every $x\in X\setminus(Y\cup H)$, $\rk_X(F\cup\{x\})\leq |H|$.
\end{lemma}
\begin{proof}
Set $n:= |H|$. If $\rk_X(F\cup\{x\})\ge n+1$, by Lemma~\ref{continue with completes caso general} there exist $n+1$ vertices $w_0,\dots,w_n\in X\setminus(F\cup\{x\})$ such that $w_i\simup{X}v$ for all $v\in F$ and $x\simup{X}w_i$.

\begin{claim}
Each $w_i$ lies in $H$.
\end{claim}

\textit{Proof of the claim:} If $w_i\in X\setminus(Y\cup H)$, choosing $v\in F\cap Y$ gives $w_i\simup{X}v$, contradicting $w_i\not\simup{X}Y$. Hence $w_i\in Y\cup H$. Since $x\in X\setminus(Y\cup H)$ and $x\simup{X}w_i$, we have $w_i\not\in Y$. Therefore $w_i\in H$.
\finishclaim

Since there are $n+1$ such $w_i\in H$, we get $|H|\ge n+1$, a contradiction.
\end{proof}

\begin{corollary}\label{rank is cero for disconnected caso general}
Let $Y<X$ be such that $w\not\simup{X} Y$ for all $w\in X\setminus Y$. Then for all $F\in\age(Y)\setminus\{\emptyset\}$ and all $x\in X\setminus Y$ we have $\rk_X(F\cup\{x\})=0$.
\end{corollary}

\begin{lemma}\label{rank for graphs with kernels caso general}
Let $X \in\sigma\cF$, $H < X$ a kernel of $X$ and $Y \leq X \setminus H$ such that every vertex $w \in X \setminus (Y \cup H)$ is not adjacent to any vertex in $Y$. Then, for any $F \in \age(Y \cup H)$ with $F \cap Y \neq \emptyset$, we have $\rk_X(F) \leq \rk_{Y \cup H}(F) + |H|+1$.
\end{lemma}
\begin{proof}
Assume $Y<X\setminus H$. We proceed by induction on $\rk_{Y\cup H}(F)$.

\textbf{Case $\rk_{Y\cup H}(F)=0$:} There is an $\cF$-good type $T=((A^j_i\mid {i\in N,\;j\in c_i}),s)$ of $F$ such that $F_T$ cannot be realized in $Y\cup H$.

\textit{Subcase 1:} If $\ran(A^j_i)\cap Y\neq\emptyset$ for some $i\in N$ and $j\in c_i$, say $x\in\ran(A^j_i)\cap Y$, then any realization of $F_T$ in $X$ would require $w\in X\setminus(Y\cup H)$ adjacent to $x\in Y$.  Hence $\rk_X(F)=0$.

\textit{Subcase 2:} If $\ran(A^j_i) \cap Y = \emptyset$ for all $i\in N$ and $j\in c_i$, since realizations of $F_T$ have the form $F \cup \{w\}$ for $w \in X \setminus(Y\cup H)$, and by Lemma~\ref{rank bounded by cardinality of kernel caso general}, these have rank at most $|H|$. Thus $\rk_X(F) \leq |H|+1$.

\textbf{Inductive step:} Suppose the claim holds for all $\beta<\alpha$. Let $G\in\age(Y\cup H)$ such that $G\cap Y\neq\emptyset$ and $\rk_{Y\cup H}(G)=\alpha$.

Since $\rk_{Y \cup H}(G)\not\geq \alpha + 1$, there is an $\cF$-good type $T$ such that all realizations of $G_T$ in $Y \cup H$ have rank $<\alpha$. Any realization $J$ in $X$ lies in $\age(Y\cup H)$ or has the form $G\cup\{w\}$ for $w\in X\setminus(Y\cup H)$. By the inductive hypothesis, $\rk_X(J)\leq\rk_{Y\cup H}(J)+|H|+1<\alpha+|H|+1$ for the first case. By Lemma~\ref{rank bounded by cardinality of kernel caso general}, $\rk_X(G\cup\{w\})\leq |H|<\alpha+|H|+1$ for the second case. Hence $\rk_X(G)\le\alpha+|H|+1$.
\end{proof}

\begin{corollary}\label{rank on the total vs rank on the connected component for all finite substructures caso general}
Let $X\in \sigma\mathcal{F}$ and $Y \leq X$ such that $w\not\simup{X} Y$ for all $w\in X\setminus Y$, and let $F\in\age(Y)\setminus\{\emptyset\}$; then $\rk_X(F) = \rk_Y(F)$ or $\rk_X(F) = \rk_Y(F) + 1$.
\end{corollary}

\begin{theorem}\label{every rank with extra caso general}
For every $(\gamma, n) \in \omega_1 \times \omega$ such that $\gamma$ is a limit ordinal or $0$, there exists a structure $G_{\gamma+n}\in\sigma\cF$ such that $H_n$ is a kernel of $G_{\gamma+n}$ and:
    \begin{enumerate}[label=(\Roman*)]
        \item\label{1 de teo superprincipal} if $x,y\in H_n$ with $x,y\in X_i$ for some $i\in n$, then $x$ and $y$ are not adjacent in $G_{\gamma+n}$,
        \item\label{2 de teo superprincipal} for every $j \in \{1, \dots, n\}$ and for all sequences $\vec{x} = (x_i)_{i \in j} \in \prod_{i \in j} X_i$, we have $\rk_{G_{\gamma+n}}(\{x_i \mid i \in j\}) \geq \gamma + (n - j)$, and
        \item $\rk(G_{\gamma+n})=\gamma+n$.
    \end{enumerate}
\end{theorem}

\begin{proof}
For $\gamma=0$, $G_n:=H_n$ satisfies the conditions by Lemma~\ref{rank on the Hn graph caso general} and Corollary~\ref{exact rank of Hn caso general}. Now let $\gamma\in\w_1$ be a limit ordinal and assume the result holds for all pairs $(\beta,m)\in\gamma\times\omega$ with $\beta$ a limit ordinal or $0$. For each $n\in\omega$, define
\[
G_{\gamma+n} =\bigoplus_{H_n}\bigl\{G_{\beta+m}\mid \beta<\gamma\text{ a limit ordinal or }0,\; m>n\bigr\}.\footnote{$\bigoplus_{H_n}$ denotes the free amalgamation over the common substructure $H_n$: the remainders $G_{\beta+m}\setminus H_n$ are taken pairwise disjoint (relabeling if necessary) and glued along a single shared copy of $H_n$, keeping the original relations within each $G_{\beta+m}$ and adding no relations between distinct remainders.}
\]

\textbf{Case $n=0$:} Conditions~\ref{1 de teo superprincipal} and~\ref{2 de teo superprincipal} are trivially true. Fix $v\in G_\gamma$. Then there is some limit $\beta<\gamma$ and $m>0$ with $v\in G_{\beta+m}$. By Corollary~\ref{1.11} and the inductive hypothesis,
\[
\rk_{G_{\beta+m}}(\{v\}) \le\rk(G_{\beta+m}) = \beta+m.
\]
By Corollary~\ref{rank on the total vs rank on the connected component for all finite substructures caso general},
\[
\rk_{G_\gamma}(\{v\}) \le\rk_{G_{\beta+m}}(\{v\})+1 \le\beta+m+1<\gamma.
\]
Every prime extension of $\emptyset$ in $G_\gamma$ has rank strictly less than $\gamma$, then $\rk(G_\gamma)\le\gamma$. On the other hand, by the inductive hypothesis and Corollary~\ref{1.13}, $\beta+m=\rk(G_{\beta+m})\le\rk(G_\gamma)$ for every $\beta<\gamma$ limit and $m>n$, whence $\rk(G_\gamma)=\gamma$.

\textbf{Case $n\ge1$:} Condition~\ref{1 de teo superprincipal} follows from the inductive hypothesis: if $x,y\in X_i$ for some $i\in n$, then $x$ is not adjacent to $y$ in $G_{\beta+m}$ for every  $\beta<\gamma$ limit and $m>n$. 

We prove Condition~\ref{2 de teo superprincipal} by induction on $l=n-j$, as in Lemma~\ref{rank on the Hn graph caso general}. The inductive step is identical. For the base case $j=n$: since $\{x_i\mid i\in n\}\subseteq H_n\subseteq G_{\beta+m}\leq G_{\gamma+n}$ for every $\beta<\gamma$ limit (or $0$) and $m>n$, the inductive hypothesis applied to $(\beta,m)$ gives $\rk_{G_{\beta+m}}(\{x_i\mid i\in n\})\geq\beta+(m-n)$; Lemma~\ref{rank on substructures vs rank on the total} then yields $\rk_{G_{\gamma+n}}(\{x_i\mid i\in n\})\geq\beta+(m-n)$, and the supremum over such $(\beta,m)$ is $\gamma$.

From that, we get $\rk_{G_{\gamma+n}}(\{v\})\ge\gamma+(n-1)$ for every $v\in X_0$, hence $\rk(G_{\gamma+n})\ge\gamma+n$.

To see $\rk(G_{\gamma+n})\le\gamma+n$, suppose $\rk(G_{\gamma+n})\ge\gamma+n+1$. By Lemma~\ref{continue with completes caso general}, there exists $G\le G_{\gamma+n}$ complete in $G_{\gamma+n}$ with $|G|=n+1$ and $\rk_{G_{\gamma+n}}(G)\ge\gamma$. If $G\subseteq H_n$, there would exist $x,y\in G$ with $x,y\in X_i$ for some $i\in n$, contradicting~\ref{1 de teo superprincipal}. Thus some $v\in G$ lies outside $H_n$, hence $v\in G_{\beta+m}\setminus H_n$ for some limit $\beta<\gamma$ and $m>n$. Then Corollary~\ref{rank function of a substructure is geq}, Lemma~\ref{rank for graphs with kernels caso general}, and $\rk(G_{\beta+m})=\beta+m$ give
\[
\gamma \le\rk_{G_{\gamma+n}}(G) \le\rk_{G_{\gamma+n}}(\{v\}) \le\rk_{G_{\beta+m}}(\{v\})+|H_n|+1 \le(\beta+m)+|H_n|+1<\gamma,
\]
a contradiction.
\end{proof}

\begin{corollary}\label{examples FAP FEP}
The following classes satisfy FAP and FEP, and therefore have the Rank Property:
\begin{enumerate}
    \item Finite graphs.
    \item Finite directed graphs (irreflexive digraphs).
    \item Finite $k$-uniform hypergraphs for any $k\geq 2$, i.e., 
    structures $(V,E)$ with $E\subseteq [V]^k$.
    $E\subseteq [V]^{<\w}$.
    \item Edge-colored versions of any of the above, i.e., 
    structures $(V,E_1,\dots,E_r)$ where the $E_i$ partition the 
    set of related tuples into color classes.
    \item More generally, finite structures in any finite 
    relational language with no additional axioms beyond 
    irreflexivity of non-unary relations.
\end{enumerate}
\end{corollary}

\begin{proof}
In each case, FAP holds because the free amalgam introduces no new relations between elements of distinct parts, and FEP holds because one can always add a vertex fully related to all existing elements.
\end{proof}

Notable Fra\"iss\'e classes that do \emph{not} satisfy both FAP and FEP include: $K_n$-free graphs for $n\geq 3$ (FAP holds but FEP fails, since adding a vertex adjacent to all others creates $K_n$); and finite linear orders and finite tournaments (FEP holds but FAP fails). For linear orders and tournaments, we prove RP by different methods in Sections~\ref{sec linear orders} and~\ref{tournaments section} respectively.

\section{Tournaments}\label{tournaments section}

A \emph{tournament} is a structure $(T,\to)$ where for every pair $\{u,v\}$ of distinct elements, exactly one of $u\to v$ or $v\to u$ holds. In this section, $\cF$ denotes the class of finite tournaments and $\sigma\cF$ the class of countable tournaments. The class $\cF$ has strong amalgamation\footnote{\textit{Strong Amalgamation Property} (SAP): if $A,B,C \in \mathcal{F}$ with $A\embed B$ and $A\embed C$, then there exist $D\in \mathcal{F}$ and embeddings $f\colon B\embed D$, $g\colon C\embed D$ such that the respective diagram commutes and $f(B)\cap g(C)=f(A)$.}, hence in particular AP, thus all results of Section~\ref{sec rank and AP} that require AP apply. The construction follows the same scheme as Section~\ref{FAP and FEP section} --- finite universal structures, kernel amalgamation, lower bounds by induction, upper bounds by controlling the rank of leaves --- but replaces the disconnection provided by FAP with two new tools: a Cross-Piece Bound on kernel amalgamations and a rigidity property of $H_n$.

For two tournaments $A$ and $B$, define $A+B$ as the tournament on $A\sqcup B$ with the original orientations within $A$ and $B$, and $a\to b$ for every $a\in A$ and $b\in B$. More generally, $\sum_{k<\w}A_k$ is defined on $\bigsqcup_{k<\w}A_k$ with the original orientations within each $A_k$ and $a\to b$ whenever $a\in A_i$, $b\in A_j$, and $i<j$.

In a tournament $T$, the types of prime extensions of $F=\{a_0,\dots,a_{m-1}\}\in\age(T)$ are coded by subsets $S\subseteq F$ and correspond to the extensions $F_S=F\cup\{z\}$ where $z\to a_i$ if and only if $a_i\in S$. There are $2^m$ such extensions, and all are tournaments.

\begin{lemma}\label{tournament continue with prescribed type}
Let $T\in\sigma\cF$, $F\in\age(T)$, $\gamma$ a limit ordinal or $0$, and $m\in\w$. Suppose $\rk_T(F)\geq\gamma+m$, and let $S\subseteq F$ be fixed. Then there exist $g_0,\dots,g_{m-1}\in T\setminus F$, pairwise distinct, such that for every $l\in\{0,\dots,m-1\}$ and every $v\in F$,
$$g_l\to v\iff v\in S,$$
and $\rk_T(F\cup\{g_0,\dots,g_{m-1}\})\geq\gamma$.
\end{lemma}

\begin{proof}
By induction on $m$. The case $m=0$ is trivial. For $m=m_0+1$: since $\rk_T(F)\geq\gamma+m_0+1$, the prime extension of $F$ encoded by $S$ has a realization $F\cup\{g_0\}$ with $\rk_T(F\cup\{g_0\})\geq\gamma+m_0$, where $g_0\to v\iff v\in S$ for $v\in F$. Apply the inductive hypothesis to $F'=F\cup\{g_0\}$ and $S$, obtaining $g_1,\dots,g_{m_0}\in T\setminus F'$. Each $g_l$ for $l\geq 1$ satisfies $g_l\to v\iff v\in S$ for $v\in F$, and $\rk_T(F\cup\{g_0,\dots,g_{m_0}\})\geq\gamma$.
\end{proof}

We now carry out the analogous construction for tournaments. We construct recursively finite sets $(X_j)_{j\in\w}$ and tournaments $(H_n)_{n\in\w}$. Label $X_j=\{v_{j,s}\mid s\in 2^j\}$ (binary strings of length $j$), then $|X_j|=2^j$. The domain of $H_n$ is $\bigcup_{j\in n}X_j$, with $|H_n|=2^n-1$. The orientations are: (1)~for $j>i$: $v_{j,s}\to v_{i,t}$ if and only if $s(i)=1$; (2)~within the same level: $v_{j,s}\to v_{j,t}$ if and only if $s<_{\mathrm{lex}}t$. Note that $H_n\leq H_{n+1}$ for all $n$.

\begin{remark}\label{tournament covering}
For any $\vec{x}=(x_0,\dots,x_{j-1})\in\prod_{i\in j}X_i$ and any type $S\subseteq\{x_0,\dots,x_{j-1}\}$, the extension of $\vec{x}$ by the vertex $v_{j,\chi_S}\in X_j$ (where $\chi_S$ is the characteristic function of $\defset{i}{x_i\in S}$) realizes the extension of $\vec{x}$ coded by $S$. This is because $v_{j,\chi_S}\to x_i$ if and only if $\chi_S(i)=1$ if and only if $x_i\in S$, independently of which $x_i\in X_i$ is chosen.
\end{remark}

Note that, unlike the construction in Section~\ref{FAP and FEP section} where a fresh vertex is created for each sequence $\vec{x}$ and each type, here the vertex $v_{j,\chi_S}$ realizes type $S$ for \emph{every} choice of representatives \mbox{$(x_i)_{i\in j}\in\prod_{i\in j}X_i$} simultaneously. This is what makes the tournament construction thus compact: $|H_n|=2^n-1$, compared to the much larger universal structures of Section~\ref{FAP and FEP section}. This reuse of vertices is possible because tournament types are determined by a subset $S\subseteq F$, which depends only on the levels of the elements, not on the specific representatives chosen. For relations of arity $\geq 3$, types involve tuples from multiple levels and this simultaneous realization is no longer possible, which is why the general construction requires fresh vertices.

\begin{lemma}\label{tournament lower bound Hn}
For every $n\geq 1$, for all $j\leq n$ and all $(x_i)_{i\in j}\in\prod_{i\in j}X_i$, we have $\rk_{H_n}(\{x_i\mid i\in j\})\geq n-j$. In particular, $\rk(H_n)\geq n$.
\end{lemma}
\begin{proof}
The proof is identical to that of Lemma~\ref{rank on the Hn graph caso general}, using Remark~\ref{tournament covering} in place of property~(5) of the construction in Section~\ref{FAP and FEP section}.
\end{proof}

\begin{theorem}\label{tournament rank Hn}
For every $n\geq 1$, $\rk(H_n)=n$.
\end{theorem}
\begin{proof}
The lower bound is Lemma~\ref{tournament lower bound Hn}. For the upper bound, we show by downward induction on $j$ that $\rk_{H_n}(F)\leq n-j$ for every $F\in\age(H_n)$ with $|F|=j$. For $j=n$: $F$ has $2^n$ prime extensions in $\mathcal{F}$, but only $|H_n|-n=2^n-1-n<2^n$ vertices remain, each extending $F$ in exactly one prime extension, then some prime extension of $F$ in $\mathcal{F}$ is unrealizable and $\rk_{H_n}(F)=0$. For $j<n$: every realization of a prime extension of $F$ has size $j+1$, thus by the inductive hypothesis its rank is at most $ n-(j+1)$, giving $\rk_{H_n}(F)\leq n-j$. With $j=0$: $\rk(H_n)\leq n$.
\end{proof}

One key combinatorial fact replacing the disconnection provided by FAP is the following:
\begin{lemma}[Cross-Piece Lemma]\label{tournament cross piece}
Let $T=A+B$ be a tournament. If $F\cap A\neq\emptyset$ and $z\in B\setminus F$, then $\rk_T(F\cup\{z\})=0$.
\end{lemma}
\begin{proof}
Let $a\in F\cap A$. The prime extension of $\{a,z\}$ coded by $S=\{a\}$ (i.e., ``$w\to a$ and $z\to w$'') is unrealizable: for $w\in A$, $z\to w$ fails (since $w\to z$ by the sum); for $w\in B$, $w\to a$ fails (since $a\to w$ by the sum). Thus $\rk_T(\{a,z\})=0$, and by Corollary~\ref{rank function of a substructure is geq}, $\rk_T(F\cup\{z\})=0$.
\end{proof}

\begin{lemma}[Localization Lemma]\label{tournament localization}
Let $T=A+B$ be a tournament. For every $F\subseteq A$ with $|F|\geq 1$,
$$\rk_T(F)\leq\rk_A(F)+1.$$
\end{lemma}
\begin{proof}
By induction on $\alpha=\rk_A(F)$. Since $\rk_A(F)=\alpha$, there exists an $S_0\subseteq F$ such that every $F\cup \{z\}$ prime extension of $F$ coded by $S_0$ and realized in $A$ has rank strictly below $\alpha$. If $z\in A$: $\rk_A(F\cup\{z\})<\alpha$, and by the inductive hypothesis, $\rk_T(F\cup\{z\})\leq\rk_A(F\cup\{z\})+1<\alpha+1$. If $z\in B$: $\rk_T(F\cup\{z\})=0$ by Lemma~\ref{tournament cross piece}. thus $\rk_T(F)\leq\alpha+1$.
\end{proof}

\begin{corollary}\label{tournament rank of sums}
For any two tournaments $A$ and $B$, $\rk(A+B)\leq\max(\rk(A),\rk(B))+2$.
\end{corollary}
\begin{proof}
Let $\delta=\max(\rk(A),\rk(B))$. For $v\in A$: $\rk_{A+B}(\{v\})\leq\rk_A(\{v\})+1\leq\delta+1$ by Lemma~\ref{tournament localization} and Corollary~\ref{1.11}. Similarly for $v\in B$. Thus $\rk(A+B)\leq\delta+2$.
\end{proof}

\begin{lemma}\label{tournament finite concatenation}
Let $\gamma$ be a limit ordinal. If $M_0,\dots,M_r$ are tournaments with $\rk(M_i)<\gamma$ for all $i\leq r$, then $\rk(M_0+\dots+M_r)<\gamma$.
\end{lemma}
\begin{proof}
By induction on $r$: for $r+1$, Corollary~\ref{tournament rank of sums} gives $\rk(M_0+\dots+M_{r+1})\leq\max(\rk(M_0+\dots+M_r),\rk(M_{r+1}))+2<\gamma$.
\end{proof}

The kernel amalgamation for tournaments plays the role of the free amalgamation over $H_n$ in Section~\ref{FAP and FEP section}.

\begin{mydef}\label{tournament kernel amalgamation}
Let $H\in\cF$ be a tournament and $\{T_k\}_{k<\w}$ a countable family of countable tournaments, each containing $H$ as a sub-tournament. The \emph{kernel amalgamation} $\bigoplus_H\{T_k\}_{k<\w}$ is the tournament on $H\cup\bigsqcup_{k<\w}(T_k\setminus H)$ with the original orientation within each $T_k$, and $u\to w$ for $u\in T_j\setminus H$ and $w\in T_k\setminus H$, with $j<k$. We call the tournaments $T_k\setminus H$ for $k\in\omega$ the \emph{leaves} of the kernel amalgamation.
\end{mydef}

\begin{lemma}[Cross-Piece Bound]\label{tournament cross piece bound}
Let $T=\bigoplus_H\{T_k\}_{k<\w}$ be a kernel amalgamation. For every $a\in T_k\setminus H$ and every $x\in T_{k'}\setminus H$ with $k\neq k'$,
$$\rk_T(\{a,x\})\leq|H|.$$
\end{lemma}
\begin{proof}
Without loss of generality $k<k'$, thus $a\to x$ in $T$. Suppose, for contradiction, that $\rk_T(\{a,x\})\geq|H|+1$. By Lemma~\ref{tournament continue with prescribed type} applied to $F=\{a,x\}$, $\gamma=0$, $m=|H|+1$, and $S=\{a\}$, there exist pairwise distinct $w_0,\dots,w_{|H|}\in T\setminus\{a,x\}$ with $w_l\to a$ and $x\to w_l$ for every $l$.
We claim $w_l\in H$ for every $l$. Let $w$ be any of these witnesses. We rule out each location outside $H$:
\begin{itemize}
\item If $w\in T_j\setminus H$ with $j<k$: then $w\to x$ by the kernel convention, contradicting $x\to w$.
\item If $w\in T_k\setminus(H\cup\{a\})$: then $w\to x$ by the kernel convention, contradicting $x\to w$.
\item If $w\in T_j\setminus H$ with $k<j<k'$: then $a\to w$ by the kernel convention, contradicting $w\to a$.
\item If $w\in T_{k'}\setminus(H\cup\{x\})$: then $a\to w$ by the kernel convention, contradicting $w\to a$.
\item If $w\in T_j\setminus H$ with $j>k'$: then $a\to w$ by the kernel convention, contradicting $w\to a$.
\end{itemize}
So $w\in H$. But then $|H|+1$ distinct elements lie in $H$, which is impossible.
\end{proof}

\begin{lemma}[Tournament Kernel Bound]\label{tournament kernel bound}
Let $T=\bigoplus_{H}\{T_k\}_{k<\w}$ be a kernel amalgamation. If $F\in\age(T_k)$ with $F\cap(T_k\setminus H)\neq\emptyset$, then $\rk_T(F)\leq\rk_{T_k}(F)+|H|+1$.
\end{lemma}

\begin{proof}
By induction on $\alpha=\rk_{T_k}(F)$.
\\
 Let $a\in F\cap(T_k\setminus H)$.
\textbf{Case $\alpha=0$:} There is an $S\subseteq F$ such that $F_S$ is not realized in $T_k$. Suppose $\rk_T(F)\geq|H|+2$. Then $F_S$ has a realization $F\cup\{w\}$ in $T$ with $\rk_T(F\cup\{w\})\geq|H|+1$. Since $F_S$ is not realized in $T_k$, we have $w\notin T_k$, thus $w\in T_{k'}\setminus H$ for some $k'\neq k$ (since $T=T_k\cup\bigsqcup_{k''\neq k}(T_{k''}\setminus H)$). By Lemma~\ref{tournament cross piece bound} and Corollary~\ref{rank function of a substructure is geq},
$$\rk_T(F\cup\{w\})\leq\rk_T(\{a,w\})\leq|H|,$$
a contradiction.
\\
\textbf{Inductive step:} Suppose the conclusion holds for all $\beta<\alpha$, with $\alpha\geq 1$. There is an $S\subseteq F$ such that every realization of $F_S$ in $T_k$ has rank in $T_k$ strictly less than $\alpha$. Suppose $\rk_T(F)\geq\alpha+|H|+2$. Then $F_S$ has a realization $F\cup\{w\}$ in $T$ with $\rk_T(F\cup\{w\})\geq\alpha+|H|+1$.
Since $T=T_k\cup\bigsqcup_{k''\neq k}(T_{k''}\setminus H)$ and $H\subseteq T_k$, we split into two cases.
\\
\textit{Case: $w\in T_k$.} Then $F\cup\{w\}\in\age(T_k)$ with $\rk_{T_k}(F\cup\{w\})<\alpha$, and by the inductive hypothesis,
$$\alpha+|H|+1\leq \rk_T(F\cup\{w\})\leq\rk_{T_k}(F\cup\{w\})+|H|+1<\alpha+|H|+1,$$
a contradiction.
\\
\textit{Case $w\in T_{k'}\setminus H$ with $k'\neq k$.} By Lemma~\ref{tournament cross piece bound} and Corollary~\ref{rank function of a substructure is geq},
$$\rk_T(F\cup\{w\})\leq\rk_T(\{a,w\})\leq|H|<\alpha+|H|+1,$$
a contradiction.
So $\rk_T(F)\leq\alpha+|H|+1$.
\end{proof}

\begin{mydef}\label{tournament construction}
For each $\alpha<\w_1$, write $\alpha=\gamma+n$ with $\gamma$ limit or $0$ and $n\in\w$. Define $T_\alpha$ recursively:
\begin{itemize}
    \item $T_n:=H_n$ for $n\in\w$.
    \item For $\gamma>0$ a limit ordinal: fix an increasing sequence $(\gamma_k)_{k<\w}$ cofinal in $\gamma$, and set \mbox{$T_\gamma:=\sum_{k<\w}T_{\gamma_k}$}.
    \item For $\gamma$ a limit ordinal and $n\geq 1$: $T_{\gamma+n}:=\bigoplus_{H_n}\bigl\{T_{\beta+m}\mid \beta<\gamma\text{ a limit ordinal or }0,\; m>n\bigr\}$.
\end{itemize}
\end{mydef}

Note that the successor case $T_{\gamma+n}$ is defined exactly as in Theorem~\ref{every rank with extra caso general}. The limit case differs from the one in Section~\ref{FAP and FEP section} in its outcome rather than its mechanism: both use the kernel amalgamation \mbox{$\bigoplus_{H_0}$} with empty kernel (the directed sum \mbox{$\sum_{k<\w}T_{\gamma_k}$} is precisely \mbox{$\bigoplus_\emptyset\{T_{\gamma_k}\}_{k<\w}$} in the tournament setting). In the FAP setting, this produces pairwise disconnected pieces; for tournaments, the kernel convention forces directed edges between distinct pieces. The following two lemmas control the rank at limit stages.

\begin{lemma}\label{0:lemma argumento de triangulo v2}
Let $\gamma$ be a limit ordinal, $n\geq 1$, and $w\in T_{\gamma+n}\setminus H_n$. There exist $m>n$ and a substructure $H\leq T_{\gamma+n}$ isomorphic to $H_m$ over $H_n$, with $w\in H$.
\end{lemma}
\begin{proof}
By the recursive construction of $T_{\gamma+n}$, there exists a finite sequence $0=\beta_0<\dots<\beta_l=\gamma$ and natural numbers $n=k_l<\dots<k_0$ such that $w$ belongs to nested copies $ T_{\beta_0+k_0}\leq T_{\beta_1+k_1}\leq\dots\leq T_{\beta_l+k_l}$ with $H_n \leq T_{\beta_0+k_0}$, where each kernel amalgamation is taken over $H_{k_i}\supseteq H_n$. Set $H=H_{k_0}\leq T_{\beta_0+k_0}\leq T_{\gamma+n}$.
\end{proof}

\begin{lemma}\label{tournament rigidity corollary}
Let $T_{\gamma+n}=\bigoplus_{H_n}\{T_{\beta+m}\mid\beta<\gamma\text{ a limit ordinal or }0,\;m>n\}$ with $\gamma$ a limit ordinal and $n\geq 1$, and let $F\subseteq H_n$ such that it contains two distinct elements $v_p,v_q$ in the same level $X_j$. Set $S=\{v_p\}$. Then every realization of $F_S$ in $T_{\gamma+n}$ lies inside $H_n$.
\end{lemma}

\begin{proof}
Suppose $F\cup\{w\}$ realizes $F_S$ with $w\notin H_n$. By Lemma~\ref{0:lemma argumento de triangulo v2}, there exist $m>n$ and $H\leq T_{\gamma+n}$ isomorphic to $H_m$ via an isomorphism that is the identity on $H_n$, with $w\in H$. Since $w\in H\setminus H_n$, the element $w$ lies in level $X_{j'}$ for some $n\leq j'\leq m-1$. As $j<n\leq j'$, by Remark~\ref{tournament covering}, $w\to v_p\iff w\to v_q$, contradicting that $F\cup \{w\}$ realizes $S=\{v_p\}$.
\end{proof}

\begin{lemma}[Pigeonhole in $H_n$]\label{tournament pigeonhole in Hn}
Let $T_{\gamma+n}=\bigoplus_{H_n}\{T_{\beta+m}\mid\beta<\gamma\text{ a limit ordinal or }0,\;m>n\}$, where $\gamma$ is a limit ordinal and $n\geq 1$. If $F\subseteq H_n$ with $|F|\geq n+1$ and $\rk_{T_{\gamma+n}}(F)\geq\gamma$, then either $F=H_n$ or there exists $w\in H_n\setminus F$ such that $\rk_{T_{\gamma+n}}(F\cup\{w\})\geq\gamma$.
\end{lemma}

\begin{proof}
Assume $F\neq H_n$. As $|F|\geq n+1$ and $H_n$ has $n$ levels $X_0,\dots,X_{n-1}$, by the pigeonhole principle there exist distinct $v_p,v_q\in F$ lying in the same level $X_j$. Set $S=\{v_p\}$. By Lemma~\ref{tournament rigidity corollary}, every realization $F\cup\{w\}$ of $F_S$ in $T_{\gamma+n}$ satisfies $w\in H_n$, hence $w\in H_n\setminus F$. Since $F\neq H_n$, the set $H_n\setminus F$ is nonempty and finite.
\\
Since $\rk_T(F)\geq\gamma$ and $\gamma$ is a limit ordinal, for every $\delta<\gamma$ we have $\rk_T(F)\geq\delta+1$; thus, the prime extension $F_S$ has a realization with rank at least $\delta$. Therefore
\[
\sup\{\rk_T(F\cup\{w\})\mid F\cup \{w\}\subseteq H_n\text{ realizes }S\}\geq\gamma.
\]
The supremum is taken over a finite set, hence some $w^*\in H_n\setminus F$ satisfies $\rk_T(F\cup\{w^*\})\geq\gamma$.
\end{proof}

\begin{theorem}\label{tournament rank of T alpha}
For every $\alpha<\w_1$, $\rk(T_\alpha)=\alpha$.
\end{theorem}
\begin{proof}
Write $\alpha=\gamma+n$ with $\gamma$ limit or $0$, and $n\in \omega$. We proceed by induction.

\textbf{Case $\gamma=0$:} $T_n=H_n$ and $\rk(H_n)=n$ by Theorem~\ref{tournament rank Hn}.

\textbf{Case $\gamma$ limit, $n=0$:} The lower bound follows from $T_{\gamma_k}\leq T_\gamma$ and the inductive hypothesis. For the upper bound, suppose that $\rk(T_\gamma)\geq\gamma+1$. Thus there exists $v$ with $\rk_{T_\gamma}(\{v\})\geq\gamma$. But $v\in T_{\gamma_k}$ for some $k\in \omega$, then we can calculate $\rk_{T_{\gamma_k}}(\{v\})\leq \gamma_k < \gamma$ by the inductive hypothesis and Corollary \ref{rank function of a substructure is geq}. Also $\rk_{T_\gamma}(\{v\}) \leq \rk_{T_{\gamma_k}}(\{v\}) +1$ by Lemma \ref{tournament kernel bound}. Hence $\rk_{T_\gamma}(\{v\}) < \gamma$, a contradiction.  

\textbf{Case $\gamma$ limit, $n\geq 1$:} The lower bound follows from $T_{\beta+m}\leq T_{\gamma+n}$ for every $\beta <\gamma$ limit or $0$ and $m>n$, and the inductive hypothesis.

For the upper bound, suppose $\rk(T_{\gamma+n})\geq\gamma+n+1$. By the definition of rank applied $n+1$ times, there exist $v_0,\dots,v_n\in T_{\gamma+n}$ with $\rk_{T_{\gamma+n}}(\{v_0,\dots,v_n\})\geq\gamma$.

\textit{Case A: some $v_l\notin H_n$.} Then $v_l\in T_{\beta+m}\setminus H_n$ for some $\beta<\gamma$ limit (or $0$) and $m>n$. By the Tournament Kernel Bound (Lemma~\ref{tournament kernel bound}) and the inductive hypothesis $\rk(T_{\beta+m})=\beta+m$,
\[
\rk_T(\{v_l\})\leq\rk_{T_{\beta+m}}(\{v_l\})+|H_n|+1\leq(\beta+m)+|H_n|+1<\gamma.
\]
By Corollary~\ref{rank function of a substructure is geq}, $\rk_T(\{v_0,\dots,v_n\})\leq\rk_T(\{v_l\})<\gamma$, contradicting $\rk_T(\{v_0,\dots,v_n\})\geq\gamma$.
\\
\textit{Case B: $\{v_0,\dots,v_n\}\subseteq H_n$.} For $n=1$, this case is vacuous since $|H_1|=1<2=n+1$, then Case A applies. Assume henceforth $n\geq 2$. Set $F_0=\{v_0,\dots,v_n\}$. We iteratively apply Lemma~\ref{tournament pigeonhole in Hn}: starting with $F_0$, at each step $k$ we have $F_k\subseteq H_n$ with $|F_k|=n+1+k$ and $\rk_T(F_k)\geq\gamma$.

Finally, by pigeonhole there are $v_p,v_q\in H_n$ in the same level. By Lemma~\ref{tournament rigidity corollary} applied to $F=H_n$ with $S=\{v_p\}$, every realization $w$ of this prime extension would satisfy $w\in H_n$. But $w\in T_{\gamma+n}\setminus H_n$ by definition of realization, then no realization exists. Hence $\rk_T(H_n)=0$, contradicting $\rk_T(H_n)\geq\gamma$.
\end{proof}

\begin{corollary}\label{tournament RP}
The class of finite tournaments has the Rank Property.
\end{corollary}

The same construction generalizes to \emph{$r$-tournaments} for any $r\geq 2$: structures over a language $\{R_1,\dots,R_r\}$ of binary relations such that for every ordered pair $(u,v)$ of distinct elements, exactly one $R_i(u,v)$ holds, and where the labels are governed by a fixed involution $\sigma:\{1,\dots,r\}\to\{1,\dots,r\}$ via the coherence condition $R_i(u,v)\iff R_{\sigma(i)}(v,u)$. Intuitively, $r$-tournaments record the outcome of each match between pairs of players, with the involution encoding how a result looks from each player's perspective. Tournaments correspond to $r=2$ with $\sigma$ swapping the two labels: $R_1(u,v)$ means ``$u$ beats $v$'' and $R_2(u,v)$ means ``$v$ beats $u$'', then $R_1(u,v)\iff R_2(v,u)$ records that ``$u$ beats $v$'' is equivalent to ``$v$ loses to $u$''. For $r=3$, one can encode \emph{win, lose, or draw}: $R_1(u,v)$ means ``$u$ beats $v$'', $R_2(u,v)$ means ``$v$ beats $u$'', $R_3(u,v)$ means ``$u$ and $v$ draw'', with $\sigma(1)=2$, $\sigma(2)=1$, $\sigma(3)=3$ (a draw looks the same from both sides). For larger $r$, finer outcome classifications are possible. The construction uses $|H_n|=(r^n-1)/(r-1)$ with $X_j=\{v_{j,s}\mid s\in r^j\}$, and the arguments adapt with appropriate choices of labels for orientations between levels and within kernel amalgamations; we leave the details to the reader.

\section{Linear orders}\label{sec linear orders}

In this section, $\cF$ denotes the class of finite linear orders and $\sigma\cF$ the class of countable linear orders. The approach here differs fundamentally from the previous two sections: since every amalgam of linear orders must totally order the elements, there is no free amalgamation and no disconnection between components. Instead, we introduce the Interval Characterization (Proposition~\ref{interval characterization}), which reduces rank computations to a one-dimensional problem, and a Pigeonhole Lemma for intervals (Lemma~\ref{pigeonhole for intervals}), which provides the upper bounds.

For two linear orders $A$ and $B$, define $A+B$ as the linear order on the disjoint union of $A$ and $B$, with the same order on each of them and $a<b$ for every $a\in A$ and $b\in B$. More generally, for a sequence $(A_k)_{k<\w}$ of linear orders, define $\sum_{k<\w}A_k$ as the linear order on the disjoint union $\bigsqcup_{k<\w}A_k$, with the original order within each $A_k$ and $a<b$ whenever $a\in A_i$, $b\in A_j$, and $i<j$. If $c\in A$, denote $A_{<c}= \defset{a\in A}{a<c}$ and $A_{>c}=\defset{a\in A}{a>c}$. We say that $S\leq A$ is a tail of $A$ if there exists $c\in A$ such that $S= (A)_{>c}$. Denote as $A^*$ to the reversed ordering of $A$.

\begin{mydef}\label{def intervals}
Let $Y$ be a linear order, $m\in \omega$ and $F=\{a_0,\dots,a_{m-1}\}\in\age(Y)$ with $a_0<\dots<a_{m-1}$. The elements of $F$ partition $Y\setminus F$ into $m+1$ \emph{intervals}:
\begin{align*}
I_0(F,Y)&=\defset{y\in Y}{y<a_0},\\
I_k(F,Y)&=\defset{y\in Y}{a_{k-1}<y<a_k}
\quad\text{for }1\leq k\leq m-1,\\
I_m(F,Y)&=\defset{y\in Y}{y>a_{m-1}}.
\end{align*}
We say these are \emph{intervals induced by $F$ in $Y$}. When $F=\emptyset$, there is a single interval $I_0(\emptyset,Y)=Y$. Write $I_k(F)$ for $k\leq m$, or only say it is an interval induced by $F$ when $Y$ is clear from context.
\end{mydef}

\begin{remark}\label{prime extensions in linear orders}
In the class of finite linear orders, the prime extensions of $F=\{a_0,\dots,a_{m-1}\}\in\age(Y)$ correspond to the $m+1$ intervals induced by $F$ in $Y$ (Definition~\ref{def intervals}): for each $0\leq k\leq m$, there is a unique prime extension adding a new element at position $k$. The prime extension corresponding to position $k$ is realizable in $Y$ if and only if $I_k (F)\neq\emptyset$; in that case, $F\cup\{c\}$ is a realization for every $c\in I_k (F)$. When $F=\emptyset$, there is a single prime extension and every element of $Y$ induces a realization.
\end{remark}

The rank of finite linear orders is completely determined by their cardinality.

\begin{lemma}[Interval Size Lemma]\label{interval size lemma}
Let $Y$ be a finite linear order, $F\in\age(Y)$, and $n\in\w$. Then $\rk_Y(F)\geq n$ if and only if every interval induced by $F$ in $Y$ has at least $2^n-1$ elements.
\end{lemma}
\begin{proof}
By induction on $n$. The case $n=0$ is trivial.

$(\Rightarrow)$: Suppose $\rk_Y(F)\geq n+1$ and fix an interval $I$ of $F$. By Remark~\ref{prime extensions in linear orders}, there exists $c\in I$ with $\rk_Y(F\cup\{c\})\geq n$. By the inductive hypothesis, each of the two sub-intervals of $I$ induced by $c$ have at least $2^n-1$ elements, giving $|I|\geq (2^n-1)+1+(2^n-1)=2^{n+1}-1$.

$(\Leftarrow)$: Suppose every interval induced by $F$ has at least $2^{n+1}-1$ elements. Fix such an interval $I$ and any $c\in I$ with at least $2^n-1$ elements strictly below and strictly above it in $I$ (such $c$ exists since $|I|\geq 2^{n+1}-1$). Then all intervals induced by $F\cup\{c\}$ have at least $2^n-1$ elements, thus $\rk_Y(F\cup\{c\})\geq n$ by the inductive hypothesis. Since this holds for every prime extension, $\rk_Y(F)\geq n+1$.
\end{proof}

\begin{theorem}\label{rank of finite linear orders}
For every finite linear order $Y$,
$$\rk(Y)=\lfloor\log_2(|Y|+1)\rfloor.$$
In particular, $\rk(Y)\geq n$ if and only if $|Y|\geq 2^n-1$.
\end{theorem}
\begin{proof}
By Lemma~\ref{interval size lemma} with $F=\emptyset$: $\rk(Y)\geq n$ iff $|Y|\geq 2^n-1$ iff $n\leq\lfloor\log_2(|Y|+1)\rfloor$.
\end{proof}

We now prove that for every countable ordinal $\alpha$ there exists a countable linear order of rank $\alpha$. The key tool is a characterization of the rank of a finite substructure purely in terms of the ranks of its intervals, viewed as independent linear orders.

\begin{proposition}[Interval Characterization]\label{interval characterization}
Let $Y$ be a linear order, $F\in\age(Y)$ with size $m$, and $\alpha$ an ordinal. Then
$$\rk_Y(F)\geq\alpha\quad\text{if and only if}\quad\rk(I)\geq\alpha\text{ for every interval }I \text{ induced by }F\text{ in }Y.$$
In particular, $\rk_Y(F)=\min\{\rk(I_j(F))\mid 0\leq j\leq m\}$.
\end{proposition}
\begin{proof}
By induction on $\alpha$, we prove both directions simultaneously. The case $\alpha=0$ is trivial. For $\alpha$ a limit ordinal, both sides reduce to the statement holding for all $\beta<\alpha$, then the result follows from the inductive hypothesis.

Suppose $\alpha=\beta+1$.

$(\Rightarrow)$: Let $\rk_Y(F)\geq\beta+1$ and fix an interval $I$ induced by $F$. By Remark~\ref{prime extensions in linear orders}, there exists $c\in I$ such that $\rk_Y(F\cup\{c\})\geq\beta$. By the inductive hypothesis, every interval induced by $F\cup\{c\}$ in $Y$ has rank at least $\beta$. In particular, the two sub-intervals $I_{<c}$ and $I_{>c}$ have rank at least $\beta$. These are exactly the two intervals induced by $\{c\}$ in $I$, viewed as an independent linear order. Since both have rank at least $\beta$, the inductive hypothesis (applied to $I$ and $\{c\}$) gives $\rk_{I}(\{c\})\geq\beta$. As $\{c\}$ realizes the unique prime extension of $\emptyset$ in $I$, we conclude $\rk(I)\geq\beta+1$.

$(\Leftarrow)$: Suppose $\rk(I)\geq\beta+1$ for every interval $I$ induced by $F$. For every such interval, there exists $c_I\in I$ with $\rk_{I}(\{c_I\})\geq\beta$. By the inductive hypothesis, $\beta \leq rk(I_{<c_I}), rk(I_{>c_I})$. The intervals induced by $F\cup\{c_I\}$ in $Y$ consist of $I_{<c_I}$, $I_{>c_I}$ and all the intervals induced by $F$ different of $I$; each of them has rank at least $\beta$. By the inductive hypothesis, $\rk_Y(F\cup\{c_I\})\geq\beta$. Since every prime extension of $F$ is realized by some $c_I$, we conclude $\rk_Y(F)\geq\beta+1$.
\end{proof}

By iterating the Interval Characterization, we obtain the following useful corollary.

\begin{corollary}[$\gamma$-level Interval Bound]\label{gamma interval bound}
Let $Y$ be a linear order, $\gamma$ a limit ordinal or $0$, and $n\in\w$. If $\rk(Y)\geq\gamma+n$, then there exist $s_1<\dots<s_{2^n-1}$ in $Y$ such that all the $2^n$ intervals induced by $\{s_1,\dots,s_{2^n-1}\}$ in $Y$ have rank at least $\gamma$.
\end{corollary}
\begin{proof}
By induction on $n$. For $n=0$, the statement is that $\rk(Y)\geq\gamma$, which holds. For $n+1$: Proposition~\ref{interval characterization} gives $c\in Y$ with $\rk(Y_{<c})\geq\gamma+n$ and $\rk(Y_{>c})\geq\gamma+n$. By the inductive hypothesis, each half contains $2^n-1$ elements whose $2^n$ intervals all have rank at least $\gamma$. Together with $c$, we have a set of size $2^{n+1}-1$  inducing $2^{n+1}$ intervals with rank at least $\gamma$.
\end{proof}

Next, we study how the rank behaves under ordinal sums and reversal orderings.

\begin{proposition}\label{rank of sums}
For any two linear orders $A$ and $B$,
$$\rk(A+B)\leq\max(\rk(A),\rk(B))+1.$$
\end{proposition}
\begin{proof}
Let $\delta=\max(\rk(A),\rk(B))$ and suppose $\rk(A+B)\geq\delta+2$. By Proposition~\ref{interval characterization}, there exists $c\in A+B$ such that both intervals induced by $\{c\}$ in $A+B$ have rank at least $\delta+1$. 
If $c\in A$: $(A+B)_{<c}=A_{<c}\leq A$, 
then $\rk(A_{<c})\leq\rk(A)\leq\delta$, a contradiction. If $c\in B$: $(A+B)_{>c}=B_{>c}\leq B$, then $\rk(B_{>c})\leq\rk(B)\leq\delta$, again a contradiction.
\end{proof}

\begin{proposition}[Reversal invariance]\label{reversal invariance}
For every linear order $Y$, $\rk(Y)=\rk(Y^*)$.
\end{proposition}
\begin{proof}
We prove $\rk(Y)\geq\alpha\Rightarrow\rk(Y^*)\geq\alpha$ by induction on $\alpha$, the reverse implication follows since $(Y^*)^*=Y$. 
The cases $\alpha=0$ and $\alpha$ limit are immediate. Let $\alpha=\beta+1$. By Proposition~\ref{interval characterization}, there exists $c\in Y$ with $\rk(Y_{<c})\geq\beta$ and $\rk(Y_{>c})\geq\beta$. Since $(Y_{<c})^*=(Y^*)_{>c}$, $(Y_{>c})^*=(Y^*)_{<c}$, and by the inductive hypothesis: $\rk((Y^*)_{<c})\geq\beta$ and $\rk((Y^*)_{>c})\geq\beta$. Therefore $\rk(Y^*)\geq\beta+1$.
\end{proof}

The following lemma provides the key upper-bound tool for computing the rank of infinite linear orders.

\begin{lemma}[Pigeonhole for intervals]\label{pigeonhole for intervals}
Let $Y$ be a linear order, $\gamma$ a limit ordinal, $r\in \omega$ and $\{M_0, M_1, \dots, M_{r-1}\}$ pairwise disjoint convex\footnote{A subset $S$ of a linear order $Y$ is \emph{convex} if for every $a,b\in S$ and $y\in Y$ with $a<y<b$, we have $y\in S$.} suborders of $Y$ such that $x<y$ for every $x\in M_i$ and $y\in M_j$ with $i<j < r$. Also assume the following:
\begin{enumerate}
    \item every tail $(M_p)_{>c}$, with $p\leq r-1$ and $c\in M_p$, has rank at least $\gamma$; and
    \item every convex suborder of $Y$, that does not contain tails of any $M_p$ for $p\leq r-1$, has rank strictly below $\gamma$.
\end{enumerate}
Then $\rk(Y)<\gamma+n$ for every $n$ with $2^n>r$.
\end{lemma}
\begin{proof}
Suppose $\rk(Y)\geq\gamma+n$ with $2^n>r$. By Corollary~\ref{gamma interval bound}, there exist $s_1<\dots<s_{2^n-1}$ in $Y$ whose $2^n$ intervals $I_0,\dots,I_{2^n-1}$ all have rank at least $\gamma$. By~(2), each $I_i$ must contain a tail of some $M_p$ with $p\leq r-1$. For each $i$, let $\varphi(i)$ be the largest such $p$. Condition~(1) implies that each $M_p$ has no maximum (otherwise the empty tail would have rank $0<\gamma$), then $s_{i+1}$ exceeds all the elements of $M_{\varphi(i)}$, giving $\varphi(i)<\varphi(i+1)$. Then $\varphi:\{0,\dots,2^n-1\}\to\{0,\dots,r-1\}$ is strictly increasing, which requires $2^n\leq r$, a contradiction.
\end{proof}

We now compute the rank of countable ordinals viewed as linear orders, and deduce the Rank Property as a corollary.

\begin{remark}[Additive principality]\label{additive principal}
For every $\beta\geq 1$ and every $c<\w^\beta$, 
the tail $(\w^\beta)_{>c}$ has order type $\w^\beta$, i.e, $(\w^\beta)_{>c}\cong\w^\beta$ for ever $c<\omega^\beta$.
\end{remark}

\begin{theorem}\label{rank of omega power times m}
For every ordinal $\beta\geq 1$ and every finite $m\geq 1$,
$$\rk(\w^\beta\cdot m)=\w\cdot\beta+\lfloor\log_2 m\rfloor.$$
\end{theorem}
\begin{proof}
By induction on $\beta$, we prove the statement for all $m\geq 1$ simultaneously. Write $k=\lfloor\log_2 m\rfloor$, then $2^k\leq m<2^{k+1}$. At each stage we first establish $\rk(\w^\beta)=\w\cdot\beta$ (the case $m=1$), then handle general $m$.

\medskip
\textbf{Step 1: $\rk(\w^\beta)=\w\cdot\beta$.}
\medskip

\textit{Lower bound.} If $\beta=\delta+1$ is a successor: $\w^\delta\cdot m\leq\w^{\delta+1}$ for every $m\geq 1$, then by the inductive hypothesis, $\rk(\w^{\delta+1})\geq\rk(\w^\delta\cdot m)=\w\cdot\delta+\lfloor\log_2 m\rfloor$ for all $m$. 
Since $\{\lfloor\log_2 m\rfloor\mid m\in\omega \}$ is unbounded in $\omega$, we get $\rk(\w^{\delta+1})\geq\w\cdot\delta+n$ for all $n<\w$, hence $\rk(\w^{\delta+1})\geq\w\cdot(\delta+1)$.

If $\beta$ is a limit ordinal: $\w^\gamma\leq\w^\beta$ for every $\gamma<\beta$, then $\rk(\w^\beta)\geq\rk(\w^\gamma)=\w\cdot\gamma$ by the inductive hypothesis. 
Since 
$sup\{\w\cdot\gamma\mid\gamma<\beta\}=\w\cdot\beta$, 
we get $\rk(\w^\beta)\geq\w\cdot\beta$.

\textit{Upper bound.} Suppose $\rk(\w^\beta)\geq\w\cdot\beta+1$. By Proposition~\ref{interval characterization}, there exists $c\in\w^\beta$ with $\rk((\w^\beta)_{<c})\geq\w\cdot\beta$. But $(\w^\beta)_{<c}=c$ as an ordinal, and $c<\w^\beta$. 
Since $c<\omega^\gamma\cdot l$ for some $\gamma<\beta$ and $l\in \omega$, then $\rk(c)\leq \omega\cdot\gamma + \lfloor \log_2(l) \rfloor$ by inductive hypotheses or Theorem \ref{rank of finite linear orders} in the case $\beta=1$, thus $\rk(c)< \omega\cdot\beta$.
This contradicts $\rk((\w^\beta)_{<c})\geq\w\cdot\beta$.

\medskip
\textbf{Step 2: $\rk(\w^\beta\cdot m)=\w\cdot\beta+\lfloor\log_2 m\rfloor$.}
\medskip

\textit{Lower bound.} Since $\w^\beta\cdot 2^k\leq\w^\beta\cdot m$, it suffices to show $\rk(\w^\beta\cdot 2^k)\geq\w\cdot\beta+k$. We proceed by induction on $k$. For $k=0$: $\rk(\w^\beta)=\w\cdot\beta$ by Step~1. For $k\geq 1$: let $c$ be the first element of the $(2^{k-1})$-th copy of $\w^\beta$ in $\omega^\beta\cdot2^k$ (indexing from $0$). The left half $(\w^\beta\cdot 2^k)_{<c}$ consists of copies $0,\dots,2^{k-1}-1$ and has order type $\w^\beta\cdot 2^{k-1}$. The right half $(\w^\beta\cdot 2^k)_{>c}$ consists of the tail of copy $2^{k-1}$ (order type $\w^\beta$ by Remark~\ref{additive principal}) followed by copies $2^{k-1}+1,\dots,2^k-1$, giving order type $\w^\beta\cdot 2^{k-1}$. By the inner induction, each half has rank at least $\w\cdot\beta+(k-1)$. Proposition~\ref{interval characterization} gives $\rk(\w^\beta\cdot 2^k)\geq\w\cdot\beta+k$.

\textit{Upper bound.} The $m$ copies of $\w^\beta$ are pairwise disjoint convex suborders of $\w^\beta\cdot m$. We verify the hypotheses of Lemma~\ref{pigeonhole for intervals} with $\gamma=\w\cdot\beta$ and $r=m$. Condition~(1): every tail of a copy has order type $\w^\beta$ (Remark~\ref{additive principal}), hence rank $\gamma$ by Step~1. Condition~(2): since the copies cover $\w^\beta\cdot m$ completely, any convex suborder not containing a tail is contained in an initial segment of some copy, which is an ordinal $c<\w^\beta$ with $\rk(c)<\gamma$ by the argument in the upper bound of Step~1. Since $\gamma=\w\cdot\beta$ is a limit ordinal, Lemma~\ref{pigeonhole for intervals} gives $\rk(\w^\beta\cdot m)<\w\cdot\beta+n$ for every $n$ with $2^n>m$, i.e., $\rk(\w^\beta\cdot m)\leq\w\cdot\beta+k$.
\end{proof}

We now extend the computation to arbitrary countable ordinals.

\begin{theorem}[Rank of countable ordinals]\label{rank CNF}
Let $\alpha\geq\w$ be a countable ordinal with Cantor normal form
$$\alpha=\w^{\beta_1}\cdot c_1+\w^{\beta_2}\cdot c_2+\dots+\w^{\beta_l}\cdot c_l,$$
where $\beta_1>\beta_2>\dots>\beta_l\geq 0$ and $1\leq c_i<\w$. Then
$$\rk(\alpha)=\w\cdot\beta_1+\lfloor\log_2 c_1\rfloor.$$
In particular, the rank depends only on the leading term of the Cantor normal form.
\end{theorem}
\begin{proof}
Write $\alpha=\w^{\beta_1}\cdot c_1+\rho$ where $\rho=\w^{\beta_2}\cdot c_2+\dots+\w^{\beta_l}\cdot c_l<\w^{\beta_1}$.

\textit{Lower bound.} Since $\w^{\beta_1}\cdot c_1\leq\alpha$, Corollary~\ref{1.13} and Theorem~\ref{rank of omega power times m} give $\rk(\alpha)\geq\rk(\w^{\beta_1}\cdot c_1)=\w\cdot\beta_1+\lfloor\log_2 c_1\rfloor$.

\textit{Upper bound.} The $c_1$ copies of $\w^{\beta_1}$ (from the leading term) are pairwise disjoint convex suborders of $\alpha$. We verify the hypotheses of Lemma~\ref{pigeonhole for intervals} with $\gamma=\w\cdot\beta_1$ and $r=c_1$. Condition~(1): every tail of a copy has order type $\w^{\beta_1}$ (Remark~\ref{additive principal}), hence rank $\gamma$ by Theorem~\ref{rank of omega power times m}. Condition~(2): a convex suborder not containing a tail of any copy is contained either in a single copy $M_p\cong\w^{\beta_1}$ (where it is bounded above, hence in an initial segment of $M_p$ — an ordinal below $\w^{\beta_1}$ with rank less than $\gamma$ by Step~1 of Theorem~\ref{rank of omega power times m}), or in the remainder $\rho$, whose rank is less than $\gamma$ by the same argument since $\rho<\w^{\beta_1}$. Since $\gamma=\w\cdot\beta_1$ is a limit ordinal ($\beta_1\geq 1$), then $\rk(\alpha)\leq\w\cdot\beta_1+\lfloor\log_2 c_1\rfloor$.
\end{proof}

\begin{corollary}\label{RP for linear orders}
The class of finite linear orders has the Rank Property.
\end{corollary}

\begin{proof}
For $\alpha<\w$, any finite linear order of size $2^\alpha-1$ has rank $\alpha$ by Theorem~\ref{rank of finite linear orders}. For $\alpha\geq\w$, write $\alpha=\w\cdot\beta+n$ with $\beta\geq 1$ and $n<\w$; then $\rk(\w^\beta\cdot 2^n)=\w\cdot\beta+n=\alpha$ by Theorem~\ref{rank of omega power times m}.
\end{proof}

The class of finite linear orders does not satisfy the free amalgamation property, then the Rank Property cannot be deduced from Theorem~\ref{every rank with extra caso general}. The Interval Characterization (Proposition~\ref{interval characterization}) replaces the kernel machinery: it reduces the rank computation to controlling the ranks of intervals, just as the kernel machinery reduces it to controlling the ranks of disconnected components.

\begin{corollary}\label{rank of omega star etc}
$\rk(\w^*)=\w$ and $\rk((\w^\beta\cdot m)^*)=\w\cdot\beta+\lfloor\log_2 m\rfloor$ for every $\beta\geq 1$ and $m\geq 1$.
\end{corollary}
\begin{proof}
Immediate from Proposition~\ref{reversal invariance} and Theorem~\ref{rank of omega power times m}.
\end{proof}

The bound $\rk(A+B)\leq\max(\rk(A),\rk(B))+1$ from Proposition~\ref{rank of sums} can be sharp when $\rk(A)=\rk(B)$: taking $A=B=\w$ gives $\rk(\w\cdot 2)=\w+1$. However, it is not sharp in general: $\rk(\w+n)=\w$ for every finite $n$ (since the leading Cantor normal form term of $\w+n$ is $\w^1\cdot 1$), while the bound gives $\w+1$.

\begin{remark}\label{rank vs Hausdorff} For a well-order $\alpha\geq\w$ with Cantor normal form leading term $\w^{\beta_1}\cdot c_1$, the Hausdorff rank $\rk^H(\alpha)$ of scattered linear orders---as defined in~\cite[Definition~1.2]{montalban-equimorphism}, following~\cite[Chapter~5]{rosenstein-book}---equals $\beta_1$. Thus $$\rk(\alpha)=\w\cdot\rk^H(\alpha)+\lfloor\log_2 c_1\rfloor,$$ showing that the rank function refines the Hausdorff rank: it captures not only the ordinal height $\rk^H(\alpha)$ but also a logarithmic measure of the multiplicity of the leading term.
\end{remark}

We now compute the rank of $\bZ\cdot\alpha$ for countable ordinals $\alpha$. We write $\bZ\cdot\alpha=\bigsqcup_{j<\alpha}\bZ^{(j)}$, where $\bZ^{(j)}\cong\bZ$ denotes the $j$-th copy of $\bZ$ in the lexicographic union (so $\bZ^{(j)}$ lies entirely below $\bZ^{(j')}$ in $\bZ\cdot\alpha$ whenever $j<j'$), and identify each $\bZ^{(j)}$ with its image in $\bZ\cdot\alpha$.

    For every linear order $Y$ and $c\in Y$, we call $c$ a cut of $Y$ and $(Y)_{<c}, (Y)_{>c}$ the left and right sides respectively.

\begin{lemma}\label{rank Z finite}
For every finite $m\geq 1$, $\rk(\bZ\cdot m)=\w+\lfloor\log_2(m+1)\rfloor$.
\end{lemma}
\begin{proof}
We use the auxiliary function $h(m)=\rk(\w+\bZ\cdot m+\w^*)$ for $m\geq 0$, with $h(0)=\rk(\w+\w^*)=\w$ (every cut $c$ in $\w+\w^*$ has one finite side).

A cut at position $c$ in the $j$-th copy of $\bZ$ (with $0\leq j\leq m-1$) in $\w+\bZ\cdot m+\w^*$ gives:
\begin{itemize}
    \item left side $\cong\w+\bZ\cdot j+\w^*$ (rank $h(j)$),
    \item right side $\cong\w+\bZ\cdot(m-j-1)+\w^*$ (rank $h(m-j-1)$).
\end{itemize}
Cuts in the initial $\w$ or the final $\w^*$ have one finite side. By Proposition~\ref{interval characterization},
$$h(m)=\left(\max_{0\leq j\leq m-1}\min(h(j),h(m-j-1))\right)+1.$$
Since $h$ is non-decreasing (as $j\leq k$ gives $\w+\bZ\cdot j+\w^*\leq\w+\bZ\cdot k+\w^*$), $\min(h(j),h(m-j-1))=h(\min(j,m-j-1))$, and the map $j\mapsto\min(j,m-j-1)$ on $\{0,\dots,m-1\}$ attains its maximum at $j=\lfloor(m-1)/2\rfloor$. Hence $h(m)=h(\lfloor(m-1)/2\rfloor)+1$ for $m\geq 1$. By induction on $m$ with base $h(0)=\w=\w+\lfloor\log_2 1\rfloor$, using the identity $\lfloor(m-1)/2\rfloor+1=\lfloor(m+1)/2\rfloor$ and the 
binary-logarithm identity $\lfloor\log_2\left( \frac{n}{2}\right)\rfloor+1=\lfloor\log_2 n\rfloor$ for $n\geq 2$ (applied with $n=m+1$): $h(m)=\w+\lfloor\log_2\left(\frac{m+1}{2}\right)\rfloor+1=\w+\lfloor\log_2(m+1)\rfloor$.

We now show $\rk(\bZ\cdot m)=h(m)$. A cut in the $j$-th copy of $\bZ$ in $\bZ\cdot m$ gives left side $\bZ\cdot j+\w^*\cong(\w+\bZ\cdot j)^*$ and right side $\w+\bZ\cdot(m-j-1)\cong(\bZ\cdot(m-j-1)+\w^*)^*$. By Proposition~\ref{reversal invariance}, the left side has rank $\rk(\w+\bZ\cdot j)$ and the right side has rank $\rk(\bZ\cdot(m-j-1)+\w^*)$.

We claim $\rk(\w+\bZ\cdot k)=\rk(\bZ\cdot k+\w^*)=h(k)$ for all $k\geq 0$. By reversal, $\rk(\w+\bZ\cdot k)=\rk(\bZ\cdot k+\w^*)$, then it suffices to show $\rk(\w+\bZ\cdot k)=h(k)$. A cut in the $j$-th copy of $\bZ$ in $\w+\bZ\cdot k$ gives left side $\w+\bZ\cdot j+\w^*$ (rank $h(j)$) and right side $\w+\bZ\cdot(k-j-1)$. Cuts in $\w$ have finite left side. Thus $\rk(\w+\bZ\cdot k)=\left(\max_j\min(h(j),\rk(\w+\bZ\cdot(k-j-1)))\right)+1$. By induction on $k$ (with base $\rk(\w)=\w=h(0)$), if $\rk(\w+\bZ\cdot l)=h(l)$ for $l<k$, then $\rk(\w+\bZ\cdot k)=\left(\max_j\min(h(j),h(k-j-1))\right)+1=h(k)$.

Therefore $\rk(\bZ\cdot m)=\left(\max_j\min(h(j),h(m-j-1))\right)+1=h(m)$.
\end{proof}

Compare with $\rk(\w\cdot m)=\w+\lfloor\log_2 m\rfloor$ (Theorem~\ref{rank of omega power times m}). The shift from $\lfloor\log_2 m\rfloor$ to $\lfloor\log_2(m+1)\rfloor$ reflects the symmetry of $\bZ$: a cut in the $j$-th copy of $\bZ$ contributes rank $\w$ to \emph{both} sides (via $\w^*$ on the left and $\w$ on the right), effectively sharing one unit of rank between the two halves.

For infinite $\alpha$, we apply the Pigeonhole lemma directly to the copies of $\bZ\cdot\w^\beta$. A notational point: we write $1+\beta$ for the ordinal sum with $1$ on the left; for finite $\beta$, $1+\beta=\beta+1$, but for $\beta\geq\w$, $1+\beta=\beta$. The distinction matters because the correct formula involves $\w\cdot(1+\beta)$, not $\w\cdot(\beta+1)$.

\begin{theorem}\label{rank of Z times omega power}
For every ordinal $\beta\geq 1$:
\begin{enumerate}
\item[(a)] For every finite $m\geq 1$, $\rk(\bZ\cdot\w^\beta\cdot m)=\w\cdot(1+\beta)+\lfloor\log_2 m\rfloor$.
\item[(b)] For every $c\in \bZ\cdot\w^\beta$, $\rk((\bZ\cdot\w^\beta)_{>c})\geq\w\cdot(1+\beta)$.
\item[(c)] For every $j<\w^\beta$, $\rk(\bZ\cdot j+\w^*)<\w\cdot(1+\beta)$.
\item[(d)] For $c\in\bZ^{(j)}$ and $d\in\bZ^{(j')}$ with $j\leq j'<\w^\beta$, $\rk((c,d))<\w\cdot(1+\beta)$.
\end{enumerate}
\end{theorem}

\begin{proof}
By transfinite induction on $\beta\geq 1$, proving (a)--(d) simultaneously. Fix $\beta\geq 1$ and assume (a)--(d) hold for all $\beta'$ with $1\leq\beta'<\beta$ (vacuous when $\beta=1$). We proceed in the order: Claim, (c), (d), (a) for $m=1$, (b), (a) for general $m$.

\smallskip\noindent\emph{Claim.} For every $\sigma$ with $1\leq\sigma<\w^\beta$, $\rk(\bZ\cdot\sigma)+2<\w\cdot(1+\beta)$.
\\ 
\smallskip\noindent\emph{Proof of Claim.} Let $\w^\gamma\cdot c$ be the leading Cantor normal form term of $\sigma$, with $\gamma<\beta$, $c\geq 1$, and $\sigma<\w^\gamma\cdot(c+1)$. Then $\rk(\bZ\cdot\sigma)\leq\w\cdot(1+\gamma)+\lfloor\log_2(c+1)\rfloor$ by inductive hypothesis on (a) when $\gamma\geq 1$, or by Lemma~\ref{rank Z finite} when $\gamma=0$. Hence $\rk(\bZ\cdot\sigma)+2<\w\cdot(1+\gamma)+\w=\w\cdot(1+\gamma+1)\leq\w\cdot(1+\beta)$, where the last inequality uses $\gamma+1\leq\beta$.
\finishclaim

\smallskip\noindent\emph{Proof of (c).} For $j=0$, $\bZ\cdot 0+\w^*=\w^*$ has rank $\w<\w\cdot(1+\beta)$. For $j\geq 1$, $\bZ\leq\bZ\cdot j$ gives $\rk(\bZ\cdot j)\geq\rk(\bZ)=\w+1>\w=\rk(\w^*)$, then Proposition~\ref{rank of sums} yields $\rk(\bZ\cdot j+\w^*)\leq\rk(\bZ\cdot j)+1<\w\cdot(1+\beta)$ by the Claim.
\\
\smallskip\noindent\emph{Proof of (d).} It follows directly by (c) since every such interval is contained in an initial segment. 

\smallskip\noindent\emph{Proof of (a) for $m=1$: $\rk(\bZ\cdot\w^\beta)=\w\cdot(1+\beta)$.}
\\
\smallskip\noindent\textit{Lower bound.} If $\beta=\delta+1$ with $\delta=0$, Lemma~\ref{rank Z finite} gives $\rk(\bZ\cdot m)=\w+\lfloor\log_2(m+1)\rfloor\geq\w+\lfloor\log_2 m\rfloor$ for every $m\geq 1$, and $\bZ\cdot m\leq\bZ\cdot\w$ gives $\rk(\bZ\cdot\w)\geq\sup_m(\w+\lfloor\log_2 m\rfloor)=\w\cdot 2=\w\cdot(1+\beta)$. If $\beta=\delta+1$ with $\delta\geq 1$, inductive hypothesis on (a) gives $\rk(\bZ\cdot\w^\delta\cdot m)=\w\cdot(1+\delta)+\lfloor\log_2 m\rfloor$ for every $m\geq 1$, and $\bZ\cdot\w^\delta\cdot m\leq\bZ\cdot\w^{\delta+1}$ gives $\rk(\bZ\cdot\w^{\delta+1})\geq\sup_m(\w\cdot(1+\delta)+\lfloor\log_2 m\rfloor)=\w\cdot(1+\delta)+\w=\w\cdot(1+\beta)$. If $\beta$ is a limit ordinal, for every $\gamma$ with $1\leq\gamma<\beta$, inductive hypothesis on (a) gives $\rk(\bZ\cdot\w^\gamma)=\w\cdot(1+\gamma)$, and $\bZ\cdot\w^\gamma\leq\bZ\cdot\w^\beta$ gives $\rk(\bZ\cdot\w^\beta)\geq\w\cdot(1+\gamma)$; by continuity of $\gamma\mapsto 1+\gamma$ at the limit $\beta$, $\rk(\bZ\cdot\w^\beta)\geq\sup_{\gamma<\beta}\w\cdot(1+\gamma)=\w\cdot(1+\beta)$.

\smallskip\noindent\textit{Upper bound.} Suppose $\rk(\bZ\cdot\w^\beta)\geq\w\cdot(1+\beta)+1$. By Proposition~\ref{interval characterization}, some cut $c$ has $\min(\rk((\bZ\cdot\w^\beta)_{<c}),\rk((\bZ\cdot\w^\beta)_{>c}))\geq\w\cdot(1+\beta)$, hence in particular its left side has rank $\geq\w\cdot(1+\beta)$. The left side is $\bZ\cdot j+\w^*$ for some $j<\w^\beta$, of rank strictly below $\w\cdot(1+\beta)$ by (c) --- contradiction.

\smallskip\noindent\emph{Proof of (b).} The set $\alpha'=(\omega^\beta)_{>j}$ has order type $\w^\beta$ (Remark~\ref{additive principal}), then $\bigsqcup_{k\in\alpha'}\bZ^{(k)}\cong\bZ\cdot\w^\beta$ embeds in $(\bZ\cdot\w^\beta)_{>c}$, giving $\rk((\bZ\cdot\w^\beta)_{>c})\geq\rk(\bZ\cdot\w^\beta)=\w\cdot(1+\beta)$ by (a) for $m=1$.

\smallskip\noindent\emph{Proof of (a) for general $m$.}

\smallskip\noindent\textit{Lower bound.} Since $\bZ\cdot\w^\beta\cdot 2^k\leq\bZ\cdot\w^\beta\cdot m$ for $2^k\leq m$, it suffices to show $\rk(\bZ\cdot\w^\beta\cdot 2^k)\geq\w\cdot(1+\beta)+k$ by induction on $k$. For $k=0$: just established. For $k\geq 1$: index the $2^k$ copies of $\bZ\cdot\w^\beta$ as $0,1,\dots,2^k-1$. Any cut at any element of the copy of index $2^{k-1}$ gives a left side containing the copies of indices $0,\dots,2^{k-1}-1$ (hence $\bZ\cdot\w^\beta\cdot 2^{k-1}$ as a suborder) and a right side containing the tail of the copy of index $2^{k-1}$ (which contains $\bZ\cdot\w^\beta$ by the proof of (b)) followed by the copies of indices $2^{k-1}+1,\dots,2^k-1$ (hence also $\bZ\cdot\w^\beta\cdot 2^{k-1}$ as a suborder). By inductive hypothesis on $k$, each side has rank at least $\w\cdot(1+\beta)+(k-1)$, and Proposition~\ref{interval characterization} gives $\rk(\bZ\cdot\w^\beta\cdot 2^k)\geq\w\cdot(1+\beta)+k$.

\smallskip\noindent\textit{Upper bound.} The $m$ copies of $\bZ\cdot\w^\beta$ are pairwise disjoint convex suborders of $\bZ\cdot\w^\beta\cdot m$. Apply Lemma~\ref{pigeonhole for intervals} with $\gamma=\w\cdot(1+\beta)$ and $r=m$. Condition~(1): by (b). 
Condition~(2): a convex suborder not containing a tail of any copy of $\bZ\cdot\w^\beta$ is contained in a single copy; 
within $\bZ\cdot\w^\beta$ it is bounded above, hence it is either an initial segment $\bZ\cdot j+\w^*$ or a bounded interval, each of rank strictly below $\gamma$ by (c) and (d). Since $\gamma$ is a limit ordinal, Lemma~\ref{pigeonhole for intervals} gives $\rk(\bZ\cdot\w^\beta\cdot m)\leq\w\cdot(1+\beta)+\lfloor\log_2 m\rfloor$.
\end{proof}

\begin{corollary}[Rank of $\bZ\cdot\alpha$]\label{rank Z CNF}
Let $\alpha$ be a countable ordinal.
\begin{enumerate}
\item If $\alpha=m$ is finite, then $\rk(\bZ\cdot m)=\w+\lfloor\log_2(m+1)\rfloor$.
\item If $\alpha\geq\w$ has Cantor normal form leading term $\w^{\beta_1}\cdot c_1$ ( $\beta_1\geq 1$), then $\rk(\bZ\cdot\alpha)=\w\cdot(1+\beta_1)+\lfloor\log_2 c_1\rfloor$.
\end{enumerate}
In particular, for $\alpha\geq\w$ the rank depends only on the leading term of the Cantor normal form.
\end{corollary}
\begin{proof}
Part~(1) is Lemma~\ref{rank Z finite}. For part~(2), write $\alpha=\w^{\beta_1}\cdot c_1+\rho$ with $\rho<\w^{\beta_1}$.
\\
\noindent\textit{Lower bound.} $\bZ\cdot\w^{\beta_1}\cdot c_1\leq\bZ\cdot\alpha$, then $\rk(\bZ\cdot\alpha)\geq\w\cdot(1+\beta_1)+\lfloor\log_2 c_1\rfloor$ by Theorem~\ref{rank of Z times omega power}(a).
\noindent\textit{Upper bound.} The $c_1$ copies of $\bZ\cdot\w^{\beta_1}$ from the leading term are pairwise disjoint convex suborders of $\bZ\cdot\alpha$. Apply Lemma~\ref{pigeonhole for intervals} with $\gamma=\w\cdot(1+\beta_1)$ and $r=c_1$. Condition~(1): by Theorem~\ref{rank of Z times omega power}(b). Condition~(2): a convex suborder not containing a tail of any copy of $\bZ\cdot\w^{\beta_1}$ is contained either in a single copy (where, being bounded above, it is an initial segment $\bZ\cdot j+\w^*$ or a bounded interval, each of rank $<\gamma$ by Theorem~\ref{rank of Z times omega power}(c)--(d)), or in the tail piece $\bZ\cdot\rho$ with $\rho<\w^{\beta_1}$ (vacuous if $\rho=0$). In the latter case (with $\rho\geq 1$), $\rk(\bZ\cdot\rho)<\gamma$: for $\rho<\w$, $\rk(\bZ\cdot\rho)=\w+\lfloor\log_2(\rho+1)\rfloor<\w\cdot 2\leq\gamma$ by Lemma~\ref{rank Z finite} and $\beta_1\geq 1$; for $\rho\geq\w$, there exists $\beta'\geq1$ and $c' \in \omega\backslash\{0\}$ such that $\rho < \omega^{\beta'}\cdot c' < \omega^{\beta_1}$ (concretely, with $\omega^{\beta'}\cdot d$ the leading term of $\rho$, take $c'=d+1$; then $\beta'<\beta_1$). Then, by Theorem~\ref{rank of Z times omega power}(a), $\rk(\bZ\cdot\rho)\leq\w\cdot(1+\beta')+\lfloor\log_2(c')\rfloor<\w\cdot(1+\beta_1)$. Since $\gamma$ is a limit ordinal, Lemma~\ref{pigeonhole for intervals} gives $\rk(\bZ\cdot\alpha)\leq\w\cdot(1+\beta_1)+\lfloor\log_2 c_1\rfloor$.
\end{proof}

\begin{remark}\label{rank Z plateau}
The rank function $\alpha\mapsto\rk(\bZ\cdot\alpha)$ is locally constant across the gaps between consecutive powers of $\w$: it depends only on the leading Cantor normal form term of $\alpha$, then $\rk(\bZ\cdot(\w^{\beta_1}\cdot c_1+\rho))=\w\cdot(1+\beta_1)+\lfloor\log_2 c_1\rfloor$ for every $\rho<\w^{\beta_1}$. In particular, $\rk(\bZ\cdot(\w+n))=\w\cdot 2$ for every $n<\w$: once the first $\w$ copies of $\bZ$ have been accumulated, no finite number of additional copies changes the rank, and the next increment requires reaching $\bZ\cdot(\w\cdot 2)$ (whose rank is $\w\cdot 2+1$). This contrasts with the finite regime, where $\rk(\bZ\cdot m)=\w+\lfloor\log_2(m+1)\rfloor$ grows unboundedly as a step function of $m$ (Lemma~\ref{rank Z finite}), and reflects the property of $\w^{\beta_1}$ in Remark \ref{additive principal} exploited in part~(b) of Theorem~\ref{rank of Z times omega power}.
\end{remark}

The family $\{\bZ\cdot\alpha\mid\alpha<\w_1\}$ almost witnesses the Rank Property for linear orders on its own. Indeed, for every countable ordinal $\gamma\geq\w+1$ there exists $\alpha<\w_1$ with $\rk(\bZ\cdot\alpha)=\gamma$: if $\gamma=\w\cdot\delta+k$ with $\delta\geq 2$ and $0\leq k<\w$, take $\alpha=\w^{\beta_1}\cdot 2^k$ where $\beta_1$ satisfies $1+\beta_1=\delta$; if $\gamma=\w+k$ with $k\geq 1$, take $\alpha=2^k-1$. The only values not realized are the finite ordinals and $\w$ itself, which are already achieved by the finite linear orders and by $\w$ respectively.

As in the case of well-orders (Remark~\ref{rank vs Hausdorff}), the rank of $\bZ\cdot\alpha$ admits a clean expression in terms of the Hausdorff rank. For $\alpha\geq\w$ with Hausdorff rank $\rk^{H}(\alpha)=\beta_1$, the linear order $\bZ\cdot\alpha$ has Hausdorff rank $\rk^{H}(\bZ\cdot\alpha)=1+\beta_1$ (under finite condensation---collapsing each maximal block of points with only finitely many points between any two of them---each copy $\bZ^{(j)}$ forms a single class, whereas distinct copies are separated by infinitely many points, so the classes are the copies ordered as $\alpha$; one condensation step thus reduces $\bZ\cdot\alpha$ to $\alpha$, whence $\rk^{H}(\bZ\cdot\alpha)=1+\rk^{H}(\alpha)$). Substituting $\rk^{H}(\bZ\cdot\alpha)=1+\beta_1$ into Corollary~\ref{rank Z CNF} gives
$$\rk(\bZ\cdot\alpha)=\w\cdot\rk^{H}(\bZ\cdot\alpha)+\lfloor\log_2 c_1\rfloor,$$
where $c_1$ is the leading coefficient of the Cantor normal form of $\alpha$. Thus the same formula relating $\rk$ to $\rk^{H}$ holds for $\bZ\cdot\alpha$ as for well-orders, even though no $\bZ\cdot\alpha$ is itself a well-order and the values differ (for instance $\rk(\w)=\w$ while $\rk(\bZ\cdot\w)=\w\cdot 2$).

\section{Questions}\label{sec questions}

The class of finite partial orders satisfies AP---indeed SAP~\cite{schmerl-posets}---, so its Fra\"iss\'e limit (the generic poset) lies in $\sigma\cF$ with rank $\infty$. Nonetheless we do not know whether it has the Rank Property: the general method of Section~\ref{FAP and FEP section} covers classes with both FAP and FEP; partial orders have FEP (a new maximum is a full good type) but lack FAP---an amalgam must respect transitivity, so the kernel construction need not yield a valid partial order.

\begin{question}\label{Q: RP vs Fraisse}
Which Fra\"iss\'e classes have the Rank Property? In particular, does the class of finite partial orders have RP?
\end{question}

\noindent Even a partial result would be of interest: for instance, whether the countable ranks of $\sigma\cF$ are cofinal in $\w_1$. This is necessary for RP but strictly weaker---the family $\{\bZ\cdot\alpha\mid\alpha\in\w_1\}$ of Section~\ref{sec linear orders} is cofinal in $\w_1$ yet skips $\w$ and every finite value, so cofinality alone does not yield RP.

\begin{question}\label{Q: weakening FAP}
What is the weakest amalgamation condition on $\cF$ that, together with a suitable extension property, implies RP? In particular, can the free amalgamation property in Theorem~\ref{every rank with extra caso general} be replaced by the strong amalgamation property?
\end{question}

\begin{question}\label{Q: model theory of RP}
What are the model-theoretic consequences of RP for the Fra\"iss\'e limit? 
\end{question}

Here RP appears to cut across the standard model-theoretic dividing lines---stability, NIP, and simplicity; see \cite[Ch.~2]{simon-nip} for IP, NIP, and stability, and \cite[Ch.~7]{tent-ziegler} for simplicity. It is compatible with IP (finite graphs have RP, while the random graph has IP) and with NIP (finite linear orders have RP, while $(\bQ,<)$ is NIP), and likewise with simplicity (the random graph is simple) and its failure ($(\bQ,<)$ is not simple). Thus RP implies none of stability, NIP, IP, or simplicity, and it is unclear whether it corresponds to any tameness notion at all---the substance of Question~\ref{Q: model theory of RP}.

For the class of finite graphs, define $N(n)$ to be the minimum $N$ such that every $X\in\sigma\cF$ containing all graphs of cardinality $N$ as induced subgraphs satisfies $\rk(X)\geq n$. The first values are $N(1)=1$, $N(2)=2$, and $4\leq N(3)\leq|H_3|$.

\begin{question}\label{Q: N(n)}
What is the growth rate of $N(n)$ for $n\in\w$?
\end{question}

\noindent For finite linear orders the analogue is trivial: since every linear order of size $m$ contains all linear orders of size $\leq m$, Theorem~\ref{rank of finite linear orders} gives $N(n)=2^n-1$. The graph case is genuinely harder, because containing all structures of a given size is a much weaker condition than having high rank.

\subsection*{Acknowledgements}
We are deeply indebted to Wies\l{}aw Kubi\'s. The rank function studied in this paper was introduced by him jointly with Saharon Shelah, and the line of research developed here was opened to us by Kubi\'s during our visit to Prague. We thank him for sharing this material with us, for many fruitful discussions in the early stages of the project, and for his continued encouragement throughout its development. We also thank Fernando Hern\'andez, Osvaldo Guzm\'an and Assaf Rinot for reading earlier versions of this manuscript and for their valuable comments. Both authors were partially supported by the PAPIIT grant IA\,104124 and the SECIHTI grant CBF2023-2024-903. The first author was additionally supported by the Israel Science Foundation (Grant Agreement 203/22). The second author was additionally supported by the Austrian Science Fund (FWF) [10.55776/STA139] and the BMFWF -- Federal Ministry of Women, Science and Research (Ernst Mach Grant - worldwide, OeAD).

\bibliography{bib}{}
\bibliographystyle{amsplain}

\end{document}